[1994/12/01]

\documentclass[12pt]{amsart}

\usepackage{amsmath,amsthm,amsfonts,amscd,amssymb,comment,eucal,latexsym,mathrsfs,bbm, dsfont, enumerate, setspace, tikz-cd, hyperref}
\usepackage[all]{xy}
\usepackage[shortcuts]{extdash}

\newtheorem{theorem}{Theorem}[section]
\newtheorem{corollary}[theorem]{Corollary}
\newtheorem{lemma}[theorem]{Lemma}
\newtheorem{proposition}[theorem]{Proposition}

\theoremstyle{definition}
\newtheorem{definition}[theorem]{Definition}

\newtheorem{example}[theorem]{Example}

\newtheorem{claim}[theorem]{Claim}

\newtheorem{fact}[theorem]{Fact}

\newcommand{\rC}{\mathrm{C}}

\newcommand{\cA}{\mathcal{A}}
\newcommand{\cB}{\mathcal{B}}

\newcommand{\cM}{\mathcal{M}}

\newcommand{\cF}{\mathcal{F}}

\newcommand{\cH}{\mathcal{H}}

\newcommand{\bE}{\mathbb{E}}
\newcommand{\bB}{\mathbb{B}}
\newcommand{\bC}{\mathbb{C}}
\newcommand{\bF}{\mathbb{F}}

\newcommand{\bN}{\mathbb{N}}
\newcommand{\bM}{\mathbb{M}}

\newcommand{\bQ}{\mathbb{Q}}

\newcommand{\cstar}{\ensuremath{\mathrm{C}^{*}}}

\newcommand{\cstarl}[1]{\ensuremath{\cstar_{\lambda}(#1)}}

\newcommand{\too}{\longrightarrow}
\newcommand{\ip}[2]{{\left\langle{#1},{#2}\right\rangle}}

\newcommand{\unorm}[1]{\left\Vert #1 \right\Vert_{\mathrm{u}}}

\newcommand{\BH}{\ensuremath{\bB(\cH)}}

\DeclareMathOperator{\dom}{dom}

\DeclareMathOperator{\pAut}{pAut}

\DeclareMathOperator{\Sym}{Sym}
\DeclareMathOperator{\pSym}{pSym}
\DeclareMathOperator{\Homeo}{Homeo}
\DeclareMathOperator{\pHomeo}{pHomeo}

\DeclareMathOperator{\id}{id}
\DeclareMathOperator{\Ext}{Ext}

\DeclareMathOperator{\supp}{supp}

\DeclareMathOperator{\Ran}{Ran}
\DeclareMathOperator{\Span}{Span}

\DeclareMathOperator{\e}{e}

\title{Residually Finite Partial Actions and MF Fell Bundles}
\author{Timothy Rainone}
\address{Department of Mathematics, Occidental College, Los Angeles, 90041}
\email{trainone@oxy.edu}

\begin{document}

\begin{abstract}
We study Blackadar and Kirchberg's matricial field (MF) property and quasidiagonality in cross-sectional \cstar-algebras constructed from Fell Bundles and, in particular, from partial \cstar-dynamical systems. In doing so we generalize Kerr and Nowak's notion of a residually finite action to partial topological dynamical systems. We look at some examples exhibiting this property including the partial Bernoulli shift which produces an MF reduced crossed product provided the group in question is exact, residually finite, and admits an MF reduced group  \cstar-algebra.
\end{abstract}

\maketitle

\section{Introduction}

There are various notions of finite-dimensional approximation in the theory of \cstar-algebras (\cite{BrownOzawaBook}). While nuclearity is expressed in terms of completely positive self maps factoring through finite-dimensional subalgebras, residual finite-dimensionality (RFD), quasidiagonality (QD), and the matricial field (MF) property (admitting norm microstates) all require external approximating maps into matrix algebras modeling the linear, multiplicative, and metric structure of the algebra. In a sense, these notions are considered more topological in nature. In this paper we are interested in such  approximation properties witnessed in \cstar-crossed products arising from partial dynamical systems, and more generally, from reduced cross-sectional \cstar-algebras that are constructed from Fell bundles over discrete groups.

The construction of a crossed product arising from a single partial automorphism of a \cstar-algebra was introduced by Exel in~\cite{ExCircleActions}. Thereafter, McClanahan defined partial \cstar-dynamical systems and their crossed products in the case where the acting group is discrete~(\cite{Mc95}). Exel then extended these constructions to the very general case of a twisted partial action of a locally compact group on a \cstar-algebra(\cite{ExTwistedPartial}). Throughout this development  several \cstar-algebras have been described as crossed products of commutative algebras by partial actions. These include AF-algebras, Bunce-Deddens algebras, and Toeplitz algebras of quasi-lattice groups (see~~\cite{ExBDpartial}, \cite{ExAFpartial},~\cite{ExLaQuPartial}). The theory of partial actions of groups on \cstar-algebras and their crossed products generalize that of \cstar-dynamical systems.  An excellent introduction to these topics can be found in~\cite{ExPAFB}. Therein, the construction of the crossed product \cstar-algebra arising from a partial \cstar-system is accomplished by means of its associated Fell bundle. A Fell bundle (or \cstar-algebraic bundle as referred to in~\cite{FellDoran} or~\cite{FellC*-bundles}) consists of a collection Banach spaces $\{B_t\}_{t\in G}$ indexed over the group $G$ whose total space $\sqcup_{t\in G}B_t$ admits multiplication and involution
with properties resembling those of a $\cstar$-algebra. Fell bundles incorporate the theory of actions, partial actions, and even twisted partial actions of groups on \cstar-algebras. To each bundle $\cB$ we can attach two \cstar-algebras; the full and reduced cross-sectional algebras denoted by $\cstarl{\cB}$, and $\cstar(\cB)$ respectively. If $\alpha$ is an action (or partial action) of $G$ on a \cstar-algebra $A$, there is a natural Fell bundle $\cA_\alpha$ and we have $\cstarl{\cA_\alpha}=A\rtimes_\lambda^\alpha G$ and $\cstar(\cA_\alpha)=A\rtimes^\alpha G$. These are studied extensively in~\cite{FellDoran} and~\cite{ExPAFB}.

The notion of amenability for Fell bundles and its relation to the nuclearity of their cross-sectional \cstar-algebras has very recently been studied. Indeed, building on Exel's approximation property (AP) for Fell bundles (\cite{ExFellAmen}), Abadie, Buss, and Ferraro show that a Fell bundle $\cB$ is amenable (has (AP)) if and only if the cross-sectional \cstar-algebra $\rC_\lambda(\cB)$ is nuclear, provided of course that the unit fiber $B_e$ is itself nuclear (\cite{AbBuFeAmenFell}). It now seems fitting, therefore, to look at various topological notions of finite-dimensional approximation in Fell bundles and their algebras, the most general of these being the matricial field property. Matricial field (MF) algebras were introduced by Blackadar and Kirchberg in~\cite{BK97}. These are stably finite \cstar-algebras constructed as generalized inductive limits of finite-dimensional algebras. The MF property is the \cstar-analogue of admitting tracial microstates, i.e., embeddability into the ultrapower $R^\omega$ of the hyperfinite II$_{1}$ factor. MF algebras are interesting in their own right but also have important connections to Voiculescu's seminal study of topological free entropy dimension~\cite{VoFreeEntropy} and with Brown-Douglas-Fillmore \emph{Ext} semigroup introduced in~\cite{BDF}. 

A natural place to begin this study is in the realm of dynamical systems. In~\cite{KerrNowak} it was shown that a continuous action $\bF_r\curvearrowright X$ of a free group on the Cantor set produces an MF reduced crossed product $\rC(X)\rtimes_\lambda\bF_r$ if and only if the action is residually finite as defined in that piece. In the classical setting residual finiteness is equivalent to the action being pseudo-nonwandering or chain recurrent in the sense of Conley (\cite{ConCR}). Residually finite actions $G\curvearrowright X$ induce quasidiagonal actions $G\curvearrowright\rC(X)$ which in turn characterize quasidiagonality in the crossed product when the acting group $G$ is amenable. In subsequent work, notions of RFD and MF actions were introduced and shown to describe those very properties in the crossed product (\cite{Rain2014},\cite{RaSc2019}). In this project we continue this work for partial dynamical systems and arrive at similar results.

We end this introduction by briefly outlining the content of this paper. In section~\ref{sec: preliminary} we introduce notation and relevant constructions that will appear thereafter. Readers familiar with the theory of Fell bundles and partial \cstar-systems may very well skip this section. Next, we define and study residually finite partial actions in section~\ref{sec: RF actions} and look at the partial Bernoulli shift. In the final section~\ref{sec: approx of Fell bundles} we define RFD, QD, and MF bundles and characterize these properties in their cross-sectional \cstar-algebras. Also, we unpack these properties in the special case of partial \cstar-systems.

The author is especially grateful to  Alcides Buss and Christopher Schafhauser for helpful discussions.

\section{Preliminaries and Notation}\label{sec: preliminary}

Throughout, $G$ will denote a discrete group with neutral element $e$. The free group on $r$ generators is written as $\bF_r$. The left-regular representation of $G$ is the unitary representation 
\[\lambda^G:G\to\bB(\ell_2(G));\quad \lambda_s^G(\xi)(t)=\xi(s^{-1}t),\ s,t\in G,\ \xi\in\ell_2(G).\]  
The reduced group \cstar-algebra of $G$; the \cstar-algebra generated by the family of unitaries $\{\lambda_s^G\}_{s\in G}$, is denoted by $\cstarl{G}$. The full group \cstar-algebra is $\cstar(G)$.

If $Z$ is any nonempty set we write $\{\epsilon_z\}_{z\in Z}$ for the canonical orthonormal basis in $\ell_2(Z)$, and for $x,y\in Z$, we write $\e_{x,y}$ for the partial isometries in $\bB(\ell_2(Z))$ defined by  $\e_{x,y}(\xi)=\ip{\xi}{\epsilon_y}\epsilon_x$.

The minimal tensor product of \cstar-algebras $A$ and $B$ is written as $A\otimes B$. Recall that a \cstar-algebra $C$ is \emph{exact} if the functor $C\otimes-$ is exact. A group $G$ is exact if and only if $\cstarl{G}$ is exact.

The algebra of $d\times d$ matrices over the complex numbers is denoted by $\bM_d$. Occasionally we will write $\mathrm{PI}(\bM_d)$ for the set of partial isometries in $\bM_d$.

Given a sequence of \cstar-algebras $(M_n)_{n\geq1}$ (usually matrix algebras in this piece), the $\ell_\infty$-product \cstar-algebra $\prod_{n\geq1}M_{n}$ contains the space
\[\bigoplus_{n\geq1}M_n:=\bigg\{(a_n)_n\in\prod_{n\geq1}M_n\mid\lim_{n\to\infty}\|a_n\|=0\bigg\}\]
as a closed, two-sided ideal. The quotient \cstar-algebra comes equipped with the canonical quotient map $\pi:\prod_{n}M_n\to\prod_{n}M_n/\oplus_{n}M_n$. Recall that
\[\|\pi\big((a_n)_n\big)\|=\limsup_{n\to\infty}\|a_n\|.\]
Given any linear space $B$ and linear map $\psi:B\to\prod_{n}M_n/\oplus_{n}M_n$ there is always a linear lift $\varphi:B\to\prod_{n}M_n$ with $\pi\circ\varphi=\psi$. If $B$ is a $\ast$-algebra we may replace $\varphi$ by $x\mapsto\frac{1}{2}(\varphi(x)+\varphi(x^*)^*)$ and assume that the lift $\varphi$ is $\ast$-linear.
Occasionally we will deal with ultraproducts. Unless stated otherwise, $\omega$ will denote a free ultrafilter on $\bN$ fixed throughout this paper. Now we may also consider the closed, two-sided ideal of $\prod_{n\geq1}M_{n}$:
\[\bigoplus_{\omega}M_n:=\bigg\{(a_n)_n\in\prod_{n\geq1}M_n\mid\lim_{n\to\omega}\|a_n\|=0\bigg\}.\]
The \emph{norm ultraproduct} of the sequence $(M_n)_{n\geq1}$ is the quotient \cstar-algebra
\[\prod_{\omega}M_n:=\prod_{n\geq1}M_n\big/\bigoplus_{\omega}M_n.\]
Here we write $\pi_\omega:\prod_{n\geq1}M_n\to\prod_{\omega}M_n$ for the quotient map. Note that
\[\pi_\omega\big((a_n)_n\big)=\lim_{\omega}\|a_n\|.\]

If $\varphi:A\to B$ is a linear map between \cstar-algebras, recall that $\varphi$ is completely positive and contractive (c.p.c.) if $\varphi\otimes\id_{\bM_n}:A\otimes\bM_n\to B\otimes\bM_n$ is positive and contractive for every $n\geq1$. If $A$ is nuclear and $\pi: P\twoheadrightarrow B$ is a quotient mapping, recall Choi and Effros' result; that any c.p.c. map $\varphi:A\to B$ admits a c.p.c. lift $\psi:A\to P$, that is, $\pi\circ\psi=\varphi$.
  
\subsection{Partial Dynamical Systems}
The notion of a partial crossed by a single automorphism was defined by R. Exel in~\cite{ExCircleActions}. Crossed products by partial actions of discrete groups was developed in~\cite{Mc95}. An excellent introduction to partial dynamical systems is~\cite{ExPAFB}. A \emph{partial dynamical system} consists of a set $X$, a group $G$, and a collection of subsets $\{U_t\}_{t\in G}$ of $X$ along with maps $\{\theta_t:U_{t^{-1}}\to U_t\}_{t\in G}$ satisfying:
\begin{enumerate}[(i)]
	\item $U_e=X$, and $\theta_e=\id_X$,
	\item $\theta_s\circ\theta_t\subseteq\theta_{st}$, for all $s,t\in G$.
\end{enumerate}
We may also say that $G$ \emph{partially acts} on $X$ and abbreviate by $$\theta:=\big\{\theta_t:U_{t^{-1}}\to U_t\big\}_{t\in G}.$$

What is meant by (ii) is that $\theta_{st}$ extends the partial composition $\theta_s\circ\theta_t$, namely
\[\theta_t^{-1}(U_t\cap U_{s^{-1}})=\dom(\theta_s\circ\theta_t)\subseteq\dom(\theta_{st})= U_{(st)^{-1}},\]
and for every $x\in \theta_t^{-1}(U_t\cap U_{s^{-1}})$ we have $\theta_s(\theta_t(x))=\theta_{st}(x)$.

A few facts follow (not-so-immediately) from the  definition. For every $s,t\in G$
\begin{enumerate}[(1)]
	\item $\theta_t:U_{t^{-1}}\to U_t$ is a bijection, with $\theta_t^{-1}=\theta_{t^{-1}}$,
	\item $\theta_s(U_{s^{-1}}\cap U_t)=U_s\cap U_{st}$,
	\item $\theta_{s^{-1}}\theta_s\theta_t=\theta_{s^{-1}}\theta_{st}$. 
\end{enumerate}
By (1) each $\theta_t$ is a partial symmetry of $X$, so we will occasionally denote a partial action succinctly as $\theta:G\to\pSym(X)$, where $\pSym(X)$ refers to the unital inverse semigroup of partial symmetries of $X$. If for each $t\in G$ we have $U_t=X$ we say that the partial action is \emph{global}. We thus recover a traditional action $G\to\Sym(X)$.

Given partial actions $\theta:G\to\pSym(X)$ and $\eta:G\to\pSym(Z)$ with partial symmetries $\{\theta_t:U_{t^{-1}}\to U_t\}_{t\in G}$, and  $\{\eta_t:V_{t^{-1}}\to V_t\}_{t\in G}$ respectively, a map $\rho:Z\to X$ is said to be \emph{$G$-equivariant} if for all $t\in G$
\[\rho(V_t)\subseteq U_t;\quad\rho(\eta_t(z))=\theta_t(\rho(z))\quad\forall t\in V_{t^{-1}}.\]
If, in addition, we have $\rho^{-1}(U_t)\subseteq V_t$, we shall call $\rho$ \emph{strictly equivariant}\footnote{This nomenclature is not standard but is useful for our purposes.}.

A \emph{continuous partial dynamical system} is a partial action of a group $G$ on a (LCH) topological space $X$; $\theta:=\big\{\theta_t:U_{t^{-1}}\to U_t\big\}_{t\in G}$, where each $U_t\subseteq X$ is open and each $\theta_t$ is a homeomorphism. We may say that $G$ acts continuously on $X$ via partial homeomorphisms and write $\theta:G\to\pHomeo(X)$.

Finally, a \emph{partial \cstar-dynamical system} consists of a partial action of a group $G$ on a \cstar-algebra $A$; $\alpha:=\big\{\alpha_t:D_{t^{-1}}\to D_t\big\}_{t\in G}$, where each $D_t\subseteq A$ is a closed ideal and each $\alpha_t$ is a $\ast$-isomorphism. We may say that $G$ acts on $A$ by partial automorphisms and abbreviate by $\alpha:G\to\pAut(A)$.

Given a continuous partial action $\theta:=\big\{\theta_t:U_{t^{-1}}\to U_t\big\}_{t\in G}$ of $G$ on $X$, we may dualize and obtain a partial \cstar-system $\alpha:=\big\{\alpha_t:D_{t^{-1}}\to D_t\big\}_{t\in G}$ where
\begin{equation}\label{eq: dual action}
D_t:=\rC_0(U_t);\quad  \alpha_t(f)=f\circ\theta_{t^{-1}},
\end{equation}
where, for an open set $U\subseteq X$ we identify $$\rC_0(U)=\{f\in\rC_0(X)\mid f(x)=0\ \forall x\notin U\}.$$
Conversely, if $\alpha:G\to\pAut(\rC_0(X))$ is a partial \cstar-system with partial automorphisms $\big\{\alpha_t:D_{t^{-1}}\to D_t\big\}_{t\in G}$, every ideal $D_t=\rC_0(U_t)$ for some open subset $U_t\subseteq X$, and using Gelfand duality there is a continuous partial system $\theta:G\to\pHomeo(X)$ satisfying~\eqref{eq: dual action}.

Starting with a partial \cstar-system of $G$ on $A$; $\alpha:=\big\{\alpha_t:D_{t^{-1}}\to D_t\big\}$ we may form the algebraic partial crossed product as follows: We begin with the linear space of all finite-supported sections:
\[\rC_c(G,A):=\big\{x:G\to A\mid x(t)\in D_t,\ \supp(x)<\infty\big\}\]
equipped with point-wise linear operations. Every such $x\in \rC_c(G,A)$ can be expressed uniquely as
\[x=\sum_{s\in G}a_s\delta_s;\]
where
\[a\delta_s:G\to A;\quad a\delta_s(t):=\begin{cases}
a 	& \text{if } t=s, \\
0 	& \text{if } t\neq s.
\end{cases}.\]
The space $\rC_c(G,A)$ comes equipped with a multiplication and involution satisfying:
\[(a\delta_s)(b\delta_t)=\alpha_s(\alpha_{s^{-1}}(a)b)\delta_{st},\quad (a\delta_s)^*=\alpha_{s^{-1}}(a^*)\delta_{s^{-1}}.\]
The resulting $\ast$-algebra is denoted by $A\rtimes_{\text{alg}}^\alpha G$. If $A$ has a unit $1_A$, then $1_A\delta_e$ is clearly a unit for $A\rtimes_{\text{alg}}^\alpha G$. Also, $A\hookrightarrow A\rtimes_{\text{alg}}^\alpha G$; $a\mapsto a\delta_e$ is a $\ast$-homomorphic embedding.

As the construction of crossed products for global actions is tied to group (unitary) representations, the partial crossed product involves partial group representations. A \emph{partial $\ast$-representation} of a group $G$ in a unital $\ast$-algebra $B$ is a a map \[v:G\to B;\quad t\mapsto v_t\]
satisfying: for all $s,t\in G$
\begin{enumerate}[(i)]\label{def: partial rep}
	\item $v_e=1$,
	\item $v_{t^{-1}}=v_t^*$,
	\item $v_{s^{-1}}v_sv_t=v_{s^{-1}}v_{st}$.
\end{enumerate}
Combining (ii) and (iii) we also obtain $v_sv_tv_{t^{-1}}=v_{st}v_{t^{-1}}$, and we see that each $v_t\in B$ is a partial isometry. Given such a partial $\ast$-representation $v:G\to B$ we will write $\e_t:=v_tv_t^*$ for the range projection of $v_t$. Clearly $\e_{t^{-1}}$ is the source projection of $v_t$. Also, it is shown that for $s,t\in G$,
\begin{equation}\label{eq: partial rep relations}
\e_s\e_t=\e_t\e_s,\quad v_s\e_t=\e_{st}v_s.
\end{equation}

Let $\alpha:G\to\pAut(A)$ be a partial \cstar-dynamical system with partial automorphisms $\{\alpha_t:D_{t^{-1}}\to D_t\}_{t\in G}$. A \emph{covariant representation} of $\alpha$ in a $\ast$-algebra $B$ is a pair $(\pi,v)$ where
\begin{enumerate}[(i)]
	\item $\pi:A\to B$ is a $\ast$-homomorphism,
	\item $v:G\to B$; $t\mapsto v_t$ is a partial $\ast$-representation, and
	\item $v_t\pi(a)v_t^*=\pi(\alpha_s(a))$ for all $t\in G$ and $a\in D_{t^{-1}}$.
\end{enumerate}
We will call such a covariant representation \emph{faithful} if $\pi$ is faithful. Given a covariant representation $(\pi,v):(A,G,\alpha)\to B$ we always have
\begin{equation}\label{eq: covariant rep relations}
\e_t\pi(a)=\pi(a)\e_t\quad\forall t\in G,\ a\in D_t.
\end{equation}
We naturally get a $\ast$-homomorphism
\[\pi\rtimes v:A\rtimes_{\text{alg}}^\alpha G\too B;\quad (\pi\rtimes v)(a\delta_s)=\pi(a)v_s.\]

Following~\cite{ExPAFB} we will define the reduced crossed product \cstar-algebra $A\rtimes_{\lambda}^\alpha G$ by means of its associated Fell bundle. We will also use Fell bundles to define the full crossed product \cstar-algebra $A\rtimes^\alpha G$.

\subsection{Fell Bundles and their \cstar-algebras.} Fell bundles appear frequently in operator algebras but perhaps most apparently in the crossed product and partial crossed product construction. Fell bundles were introduced by Fell in~\cite{FellC*-bundles}, under the name C*-algebraic bundles. Detailed treatments can be found in~~\cite{FellDoran} and~\cite{ExPAFB}.

A Fell bundle over a group $G$ is a family of Banach spaces indexed over $G$ that admits a graded multiplication and involution.

\begin{definition} A \emph{Fell bundle} over a group $G$ is a collection of Banach spaces $\cB=\{B_t\}_{t\in G}$ whose disjoint union $B:=\sqcup_{t\in G} B_t$ (called the \emph{total space} of $\cB$) admits a multiplication and involution:
\[\cdot: B\times B\rightarrow B,\quad  \ast: B\rightarrow B\] 
satisfying the following properties: for all $s,t\in G$ and $a,b\in B$:
\begin{enumerate}[(i)]
\item $B_s\cdot B_t\subseteq B_{st}$,
\item multiplication is bilinear from $B_s\times B_t$ to $B_{st}$,
\item multiplication on $B$ is associative,
\item $B_t^*\subseteq B_{t^{-1}}$,
\item  involution is conjugate-linear from $B_t$ to $B_{t^{-1}}$,
\item $(ab)^*=b^*a^*$,
\item $a^{**}=a$,
\item $\|ab\|\leq\|a\|\|b\|$,
\item $\|a^*\|=\|a\|$,
\item $\|a^*a\|=\|a\|^2$,
\item $a^*a\geq 0$ in $B_e$.
\end{enumerate}
\end{definition}

Implicit in the definition of a Fell bundle is the fact that the unit fiber $B_e$ is a \cstar-algebra. This gives meaning to condition (xi). If the algebra $B_e$ is unital, we shall call the bundle \emph{unital}. Moreover, a Fell bundle $\cB=\{B_t\}_{t\in G}$ is called \emph{separable} if $G$ is countable and each space $B_t$ is separable.

We will occasionally make use of the following fact.

\begin{fact}
	If $\cB=\{B_t\}_{t\in G}$ is a Fell bundle over a group $G$, and $(u_i)_i$ is any approximate identity for $B_e$, then
	\[\|bu_i-b\|_{B_t}\too0\quad \forall b\in B_t.\]
\end{fact}

\subsubsection{Representations of Fell Bundles} A representation of a Fell bundle $\cB=\{B_t\}_{t\in G}$ in a $\ast$-algebra $C$ is a family of linear maps $\pi:=\{\pi_t:B_t\rightarrow C\}_{t\in G}$ satisfying: for every $s,t\in G$, $a\in B_s$, $b\in B_t$
\begin{enumerate}
	\item[(i)] $\pi_s(a)\pi_t(b)=\pi_{st}(ab)$;
	\item[(ii)] $\pi_s(a)^*=\pi_{s^{-1}}(a^*)$.
\end{enumerate}
With a slight abuse of notation, we may write $\pi:\cB\to C$ to abbreviate such a representation. If $B_e$ and $C$ are unital we will say that the representation $\pi$ is \emph{unital} provided $\pi_e:B_e\to C$ is unital. 
Moreover, we will call $\pi$ \emph{faithful} if $\pi_e$ is {faithful}\footnote{Recall that a (linear) map $\phi:A\to C$ between $\ast$-algebras is faithful if $\phi(a^*a)=0$ implies $a=0$.}.

Note that if $\{\pi_t\}_{t\in G}$ is a representation of a Fell bundle $\cB=\{B_t\}_{t\in G}$ in a $\ast$-algebra $C$, then $\pi_e:B_e\to C$ is a $\ast$-homomorphism. Consequently, representations are contractive on each section.

\begin{fact}\label{fact: reps are contractive}
If $\cB=\{B_t\}_{t\in G}$ is a Fell Bundle and  $\pi:\cB\to C$ a representation in a \cstar-algebra $C$, then each $\pi_t:B_t\to C$ is contractive. If $\pi$ is faithful, each $\pi_t$ is isometric.
\end{fact}

\begin{proof}
	Since $\pi_e$ is contractive, given $a\in B_s$ we have
	\[\|\pi_s(a)\|^2=\|\pi_s(a)^*\pi_s(a)\|=\|\pi_{s^{-1}}(a^*)\pi_s(a)\|=\|\pi_e(a^*a)\|\leq\|a^*a\|=\|a\|^2.\]
	If $\pi$ is faithful the last inequality is an equality.
\end{proof}

We may compose a representation with a $\ast$-homomorphism as follows. Suppose $\cB=\{B_t\}_{t\in G}$ is a Fell Bundle,  and $A$ and $C$ are $\ast$-algebras. If $\lambda:\cB\to A$ is a representation and $\phi:A\to C$ is a $\ast$-homomorphism, then 
\[\phi\circ\lambda:=\big\{(\phi\circ\lambda)_t:=\phi\circ\lambda_t:B_t\to C\big\}_{t\in G}\]
is indeed a representation of $\cB$ in $C$.

Next, if $\{\pi_n:\cB\to C_n\}_{n\geq1}$ is a sequence of representations, their \emph{direct sum}
\begin{equation}\label{eq: direct sum rep}
\pi:=\oplus_{n\geq1}\pi_n:\cB\too\prod_{n\geq1}C_n;\quad\pi_s(a):=\big(\pi_{n,s}(a)\big)_{n\geq1},
\end{equation}
is also a representation. If each $C_n$ is a \cstar-algebra, Fact~\ref{fact: reps are contractive} ensures that~\eqref{eq: direct sum rep} is well-defined.

We will also need to tensor a bundle representation with the left-regular representation $\lambda^G$ of the group $G$. More precisely, if $\pi=\big\{\pi_t:B_t\to\BH\big\}_{t\in G}$ is a representation in $\BH$ for some Hilbert space $\cH$, it is easily verified that 
\begin{equation}\label{eq: pi tensor with lambda}
\pi\otimes\lambda^G:=\big\{(\pi\otimes\lambda^G)_t:B_t\to\bB(\cH\otimes_2\ell_2(G))\big\}_{t\in G};\quad (\pi\otimes\lambda^G)_t(b)=\pi_t(b)\otimes\lambda_t^G
\end{equation}
is also a representation.

\subsubsection{The left-regular representation} There is a canonical representation of a Fell bundle in the algebra of adjointable operators on the Hilbert \cstar-module $\ell_2(\cB)$ whose construction we now briefly describe. 

Fix a Fell bundle $\cB=\{B_t\}_{t\in G}$ over a group $G$, and let
\[\rC_c(\cB)=\bigg\{x:G\rightarrow\bigsqcup_{t\in G}B_t\ \big|\ \supp(x)<\infty,\ x(t)\in B_t\ \forall t\in G\bigg\}\]
denote the $\mathbb{C}$-linear space of all finitely supported sections with pointwise operations. If $r\in G$ with $b\in B_r$ we write
\[b\delta_r:G\to\bigsqcup_{t\in G} B_{t};\quad b\delta_r(s):=\begin{cases}
b 	& \text{if } s=r \\
0 	& \text{if } s\neq r.
\end{cases}.\]
Note that every $\xi\in \rC_c(\cB)$ can be written as a formal sum
$$\xi=\sum_{t\in G}b_t\delta_t$$  
where $\xi(t)=b_t\in B_t$ and only finitely many $b_t$ are non-zero. The linear space $X:=\rC_c(\cB)$ admits a right module action of the \cstar-algebra $B_e$ given by
\begin{equation}\label{action}\xi\cdot b(s)=\xi(s)b\qquad \xi\in X,\ b\in B_e,
\end{equation}
and a $B_e$-valued inner product $X\times X\longrightarrow B_e$
\begin{equation}\label{innerproduct}\langle\xi,\eta\rangle:=\sum_{t\in G}\xi(t)^*\eta(t).
\end{equation}
Note that $\xi(t)^*\eta(t)\in B_{t^{-1}}B_t\subset B_e$. In this way $X$ is an inner product $B_e$-module. Completing $X$ with respect to the norm
\[\|\xi\|_2=\|\langle\xi,\xi\rangle\|^{1/2}=\bigg\|\sum_{t\in G}\xi(t)^*\xi(t)\bigg\|^{1/2}\]
and extending the right action and inner product continuously produces a Hilbert $B_e$-module which we denote by $\ell_2(\cB)$. Definitions imply that for a vector $\xi=b\delta_t$, where $b\in B_t$, we have $\|\xi\|_2=\|b\|$. 

 The following is a useful fact.

\begin{lemma}\label{evaluation} Let $\cB=\{B_t\}_{t\in G}$ be a Fell Bundle over $G$. For each $s\in G$ the map $P_s:\ell_2(\cB)\rightarrow B_s$ given by $P_s(\xi)=\xi(s)$ is linear and contractive. 
\end{lemma}
\begin{proof}
	Linearity is clear. Note that for $\xi\in\ell_2(\cB)$
	\[0\leq\xi(s)^*\xi(s)\leq\sum_{t\in G}\xi(t)^*\xi(t),\]
	holds in $B_e$. Taking norms we get
	\[\|P_s(\xi)\|^2=\|\xi(s)\|^2=\|\xi(s)^*\xi(s)\|\leq \bigg\|\sum_{t\in G}\xi(t)^*\xi(t)\bigg\|=\|\langle\xi,\xi\rangle\|=\|\xi\|^2.\]
\end{proof}

We write $\mathcal{L}(\ell_2(\cB))$ for the \cstar-algebra of all adjointable maps on the Hilbert module $\ell_2(\cB)$.

\begin{proposition}\label{propleftregular}
	Let $\cB=\{B_t\}_{t\in G}$ be a Fell Bundle over $G$ with its associated inner product $B_e$-module $X$ as described above. For each $s\in G$ and $b\in B_s$ define the map $\lambda_s(b):X\rightarrow X$ by
	\begin{equation}\label{leftregular}\lambda_s(b)\left(\sum_{t\in G}b_t\delta_t\right)=\sum_{t\in G}bb_t\delta_{st}.
	\end{equation}
	\begin{enumerate}
		\item[(i)] $\lambda_s(b)$ extends continuously to a bounded linear map on $\ell_2({\cB})$ satisfying
		\begin{equation}\label{leftregular2}\lambda_s(b)\xi(t)=b\xi(s^{-1}t), \quad \xi\in\ell_2(\cB), t\in G.
		\end{equation}
		\item[(ii)] $\lambda_s(b)$ is adjointable with $\lambda_s(b)^*=\lambda_{s^{-1}}(b^*)$.
		\item[(iii)] $\|\lambda_s(b)\|=\|b\|$.
	\end{enumerate}
\end{proposition}

\begin{proof}
	It is clear that $\lambda_s(b)$ as defined by \eqref{leftregular} is linear. Since
	\begin{align*}\bigg\|\lambda_s(b)\bigg(\sum_{t\in G}b_t\delta_t\bigg)\bigg\|^2&=\bigg\|\sum_{t\in G}bb_t\delta_{st}\bigg\|^2=\bigg\|\sum_{t\in G}(bb_t)^*bb_t\bigg\|=\bigg\|\sum_{t\in G}b_t^*(b^*b)b_t\bigg\|\\&\leq\bigg\|\sum_{t\in G}\|b\|^2b_t^*b_t\bigg\|=\|b\|^2\bigg\|\sum_{t\in G}b_t^*b_t\bigg\|=\|b\|^2\bigg\|\sum_{t\in G}b_t\delta_t\bigg\|^2
	\end{align*}
	we see that $\lambda_s(b)$ extends to a bounded operator on $\ell_2(\cB)$ with norm $\|\lambda_s(b)\|\leq\|b\|$. It easily follows that equation \eqref{leftregular2} holds for $\xi\in X$. Continuity and Lemma~\ref{evaluation} ensure that \eqref{leftregular2} holds for all $\xi\in\ell_2(\cB)$.
	
	To see (ii),
	\begin{align*}
	\langle\lambda_s(b)\xi,\eta\rangle&=\sum_{t\in G}\lambda_s(b)\xi(t)^*\eta(t)=\sum_{t\in G}(b\xi(s^{-1}t))^*\eta(t)=\sum_{t\in G}\xi(s^{-1}t)^*b^*\eta(t)\\&\stackrel{r=s^{-1}t}{=}\sum_{r\in G}\xi(r)^*b^*\eta(sr)=\sum_{r\in G}\xi(r)^*\lambda_{s^{-1}}(b^*)\eta(r)=\langle\xi, \lambda_{s^{-1}}(b^*)\eta\rangle.
	\end{align*}
	
	As for (iii), we already have that $\|\lambda_s(b)\|\leq \|b\|$ for $b\in B_s$. Let $(v_i)_i$ be an approximate identity for the \cstar-algebra $B_e$. Using the fact that $v_i\delta_e$ are in the unit ball of $\ell_2(\cB)$ we see that
	\[\|\lambda_s(b)\|\geq\|\lambda_s(b)(v_i\delta_e)\|_2=\|bv_i\delta_s\|_2=\|bv_i\|_{B_s}\rightarrow\|b\|.\]
	Therefore $\|b\|\leq\|\lambda_s(b)\|\leq \|b\|$ as claimed.
\end{proof}

\begin{proposition}\label{prop: left-reg rep}
Let $\cB=\{B_t\}_{t\in G}$ be a Fell Bundle over $G$. The collection of maps \[\lambda:=\{\lambda_t:B_t\rightarrow\mathcal{L}(\ell_2(\cB))\}_{t\in G}\]
is a representation of $\cB$ in $\mathcal{L}(\ell_2(\cB))$.
\end{proposition}

\begin{proof}
	The maps $\lambda_t:B_t\rightarrow\mathcal{L}(\ell_2(\cB))$, $b\mapsto\lambda_t(b)$ are clearly linear. The fact that $\lambda_s(b)^*=\lambda_{s^{-1}}(b^*)$ for all $s\in G$, $b\in B_s$ was verified in Proposition~\ref{propleftregular}. Also, for any $s,t,r\in G$, $b\in B_s$, $b'\in B_t$, and $\xi\in\ell_2(\cB)$ we have
	\begin{align*}\lambda_s(b)\lambda_t(b')\xi(r)&=b\lambda_t(b')\xi(s^{-1}r)=bb'\xi(t^{-1}s^{-1}r)=bb'\xi((st)^{-1}r)\\&=\lambda_{st}(bb')\xi(r),
	\end{align*}
so $\lambda_s(b)\lambda_t(b')=\lambda_{st}(bb')$. Thus $\lambda$ is a representation. 
\end{proof}

The representation $\lambda:\cB\to\mathcal{L}(\ell_2(\cB))$ is called the \emph{left-regular} representation of $\cB$.

\subsubsection{The cross-sectional \cstar-algebras $\cstarl{\cB}$ and $\cstar(\cB)$} 
Fix a Fell bundle $\cB=\{B_t\}_{t\in G}$ over a group $G$, and consider again the $\mathbb{C}$-linear space $\rC_c(\cB)$ of all finitely supported sections. The space $\rC_c(\cB)$ also admits the structure of a $\ast$-algebra where multiplication and involution are defined by
\begin{align*}
x\cdot y(s)=\sum_{t\in G}x(t)y(t^{-1}s)&=\sum_{t\in G}x(st^{-1})y(t),\quad s\in G\\
x^*(s)&=x(s^{-1})^*,\quad s\in G.
\end{align*}
These formulas are natural as they extend the group operation; that is, for $s,t\in G$, $a\in B_s$, $b\in B_t$ we have
\begin{align*}
(a\delta_s)(b\delta_t)&=ab\delta_{st}\\
(a\delta_s)^*&=a^*\delta_{s^{-1}}.
\end{align*}

Note that there is a natural canonical representation of $\cB$ in $\rC_c(\cB)$ given by \[j:=\{j_t:B_t\rightarrow \rC_c(\cB)\}_{t\in G};\quad j_t(b)=b\delta_t.\]

Akin to the theory of unitary representations of groups and their induced group algebra homomorphisms, representations of Fell bundles are in one-to-one correspondence with $\ast$-homomorphisms of the $\ast$-algebra $\rC_c(\cB)$. 

\begin{proposition}\label{prop: induced reps}
Let $\cB=\{B_t\}_{t\in G}$ be a Fell bundle over a group $G$, and let $\{\pi_t\}_{t\in G}$ be a representation of $\cB$ in a $\ast$-algebra $C$. There is an induced $\ast$-homomorphism $\pi: \rC_c(\cB)\rightarrow C$
given by 
\begin{equation}\label{homo}\pi(x)=\sum_{t\in G}\pi_t(x(t)).
\end{equation}
Conversely, if $C$ is a $\ast$-algebra and $\pi: \rC_c(\cB)\rightarrow C$ a $\ast$-homomorphism,  there is a representation $\{\pi_t:B_t\rightarrow C\}_{t\in G}$ of the Fell bundle $\cB$  in $C$ given by
\begin{equation}\label{rep}\pi_t(b)=\pi(b\delta_t),\quad b\in B_t
\end{equation}
and whose induced $\ast$-homomorphism recovers $\pi$.
\begin{displaymath}
\xymatrix{
	B_t\ar[d]_{j_t} \ar[r]^{\pi_t} &C \\
	\rC_c(\cB)\ar[ur]_{\pi} & }
\end{displaymath}
\end{proposition}

We may slightly abuse notation and write $\pi=\{\pi_t:B_t\to C\}_{t\in G}$ for the representation as well as for the induced $\ast$-homomorphism  $\pi:\rC_c(\cB)\to C$.

\begin{proposition}\label{prop: lambda in 1-1}
Let $\cB=\{B_t\}_{t\in G}$ be a Fell Bundle and let $\lambda:\cB\to\mathcal{L}(\ell_2(\cB))$ be the left-regular representation as in Proposition~\ref{prop: left-reg rep}. The induced $\ast$-homomorphism $\lambda:\rC_c(\cB)\rightarrow \mathcal{L}(\ell_2(\cB))$ 
\[\lambda(x)=\sum_{t\in G}\lambda_t(x(t))\]
is injective.
\end{proposition}

\begin{proof}
Suppose $x\in \rC_c(\cB)$ and suppose $\lambda(x)=0$ in $\mathcal{L}(\ell_2(\cB))$. Let $(v_i)_i$ be an approximate identity for $B_e$. Then
\[0=\lambda(x)(v_i\delta_e)=\sum_{t\in G}\lambda_t(x(t))(v_i\delta_e)=\sum_{t\in G}(x(t)v_i)\delta_t\]
as a vector in $\ell_2(\cB)$. It follows that $x(t)v_i=0$ for all $t$ and for all $i$. Fixing a $t$ we see $0=x(t)v_i\rightarrow x(t)$. So $x(t)=0$ for all $t$ whence $x=0$.
\end{proof}

Thanks to Proposition~\ref{prop: lambda in 1-1}, given a Fell bundle $\cB$ we may now define the \emph{reduced norm} on $\rC_c(\cB)$ given by
\[\|x\|_{\lambda}=\|\lambda(x)\|\quad x\in \rC_c(\cB).\]
This is indeed a \cstar-norm and the completion: 
\[\cstar_{\lambda}(\cB):=\overline{\rC_c(\cB)}^{\|\cdot\|_{\lambda}}\]
is the \emph{reduced cross-sectional \cstar-algebra} associated to $\cB$.

Note that Proposition~\ref{propleftregular} implies that the maps $j_t:B_t\rightarrow\cstar_{\lambda}(\cB)$ given by $j_t(b)=b\delta_t$ are isometric embeddings. Indeed, definitions imply that $\lambda\circ j_t=\lambda_t:B_t\rightarrow\mathcal{L}(\ell_2(\cB))$ for every $t\in G$ and so
\[\|j_t(b)\|_\lambda=\|\lambda\circ j_t(b)\|=\|\lambda_t(b)\|=\|b\|.\]
In particular, $B_e$ can be viewed as a \cstar-subalgebra of $\cstar_\lambda(\cB)$ via the embedding $j_e$.

There is also a \cstar-algebra arising from a bundle that is universal with respect to its representations.

\begin{proposition}\label{prop: reps are bounded}
If $\pi=\{\pi_t\}_{t\in G}$ is a representation of a Fell Bundle $\cB=\{B_t\}_{t\in G}$ in a \cstar-algebra $C$, then for all $x\in \rC_c(\cB)$ we have
\[\|\pi(x)\|\leq\sum_{t\in G}\|x(t)\|.\]
\end{proposition}

\begin{proof}
Simply apply Proposition~\ref{prop: induced reps}, linearity, and Fact~\ref{fact: reps are contractive}.
\end{proof}

Given a Fell bundle $\cB$, Propositions~\ref{prop: reps are bounded} and~\ref{prop: lambda in 1-1} ensure that 
\[\|x\|_{\mathrm{u}}=\sup\bigg\{\|\pi(x)\|\ |\ \pi:\cB\rightarrow C\quad\text{any representation}\bigg\},\]
is a well-defined \cstar-norm on $\rC_c(\cB)$, referred to as \emph{universal norm}. We also can define the \emph{full cross-sectional \cstar-algebra} as the completion
\[\cstar(\cB):=\overline{\rC_c(\cB)}^{\|\cdot\|_{\mathrm{u}}}.\]
Moreover, $\cstar(\cB)$ enjoys a universal property: given any \cstar-algebra $C$ and representation $\{\pi_t:B_t\to C\}_{t\in G}$, there is a unique $\ast$-homomorphism $\pi:\cstar(\cB)\to C$ satisfying $\pi\circ j_t=\pi_t$. Consequently, there is a canonical quotient map 
\[\pi_{\lambda}:\cstar(\cB)\rightarrow\cstar_{\lambda}(\cB).\]

\subsubsection{Partial crossed-product \cstar-algebras}
Given a partial \cstar-dynamical system $\big(A,G, \alpha=\big\{\alpha_t:D_{t^{-1}}\to D_t\big\}_{t\in G}\big)$ we naturally obtain a Fell bundle. Indeed, for each $s\in G$ we may consider the Banach space
\[A_s:=\big\{(a,s)\mid a\in D_s\ \big\}=D_s\times\{s\}\subseteq A\times G\]
with linear operations and norm:
\begin{itemize}
	\item $(a,s)+(b,s)=(a+b,s)$, 
	\item $z(a,s)=(za,s)$, 
	\item $\|(a,s)\|:=\|a\|$, 
\end{itemize}
where $a,b\in A$, $z\in\mathbb{C}$. Multiplication and involution are then defined as:
\begin{itemize}
	\item $(a,s)(b,t)=(\alpha_s(\alpha_{s^{-1}}(a)b),st)$, 
	\item $(a,s)^*=(\alpha_{s^{-1}}(a^*),s^{-1})$,
\end{itemize}
where $a,b\in B$, and $s,t\in G$. We will usually write $\cA_\alpha=\big\{A_s\big\}_{s\in G}$ to denote such a bundle arising from the action $\alpha$. 

It is easily verified that we have a $\ast$-isomorphism of $\ast$-algebras
\begin{equation}\label{eq: *-iso between crossed product and bundle}
A\rtimes_{\mathrm{alg}}^{\alpha}G\too\rC_c(\cA_\alpha);\quad a\delta_s\mapsto(a,s)\delta_s.
\end{equation}
The \emph{reduced crossed product} and \emph{full crossed product} \cstar-algebras are then defined as
\[A\rtimes_{\lambda}^{\alpha}G:=\cstarl{\cA_\alpha},\quad A\rtimes_{\lambda}G:=\cstar(\cA_\alpha).\]

\subsubsection{The expectation and Fell's absorption}

Here we recall a few facts about the canonical faithful expectation \[\mathbb{E}:\cstarl{\cB}\rightarrow B_e\] 
and prove a uniqueness result.

The following can be found in the literature.
\begin{proposition}\label{expectation} Let $\cB=\{B_t\}_{t\in G}$ be a Fell bundle over a group $G$. For every $s\in G$ there is a contractive linear map
	$\bE_s:\cstar_{\lambda}(\cB)\to B_s$ satisfying 
	\begin{enumerate}[(i)]
		\item $\mathbb{E}_s(x)=x(s)$ for every $x\in \rC_c(\cB)$,
		\item $\bE_s\circ j_s=\id_{B_s}$.
	\end{enumerate}
Moreover, $\bE_e:\cstar_{\lambda}(\cB)\to B_e$ is faithful.
\end{proposition}

If we consider the spaces $B_s$ embedded in $\cstar_{\lambda}(\cB)$, then Proposition~\ref{expectation} (iii) essentially says that $\bE_s^2=\bE_s$. Thus $\bE:=\bE_e:\cstar_\lambda(\cB)\rightarrow B_e$ is a linear, contractive, faithful, and idempotent map onto the \cstar-algebra $B_e$. Such a map is our conditional expectation.

We wrap up this preliminary section by establishing Fell's absorption principle using a useful lemma.

\begin{lemma}\label{lem: factorlemma} Let $A, B, C$, and $D$ be C*-algebras. Suppose $\phi:A\rightarrow B$ and $\pi:A\rightarrow D$ are $\ast$-homomorphims with $\pi$ a quotient mapping. Suppose further that $f:B\rightarrow C$ and $g:D\rightarrow C$ are maps satisfying $g\circ\pi=f\circ\phi$. If $f$ is faithful then there is a $\ast$-homomorphism $\varphi:D\rightarrow B$ such that the following diagram commutes. 
\[\begin{tikzcd}
A \arrow[r, "\phi"]	\arrow[rd, "\pi"']	& B\arrow[r, "f"]&	C	\\
&	D \arrow[red,u,dashed, "\exists\varphi"]	\arrow[ru, "g"']	&
\end{tikzcd}
\]
If, moreover, $g$ is faithful, then $\varphi$ is an embedding.

\end{lemma}

\begin{proof}
Since $f$ is faithful we see that $\ker(\pi)\subseteq\ker(\phi)$. Indeed;
\begin{align*}\pi(a)&=0\Rightarrow\pi(a^*a)=0\Rightarrow g\pi(a^*a)=0\Rightarrow f\phi(a^*a)=0\\
&\Rightarrow f(\phi(a)^*\phi(a))=0\Rightarrow\phi(a)=0
\end{align*}
We may therefore define $\varphi: D\rightarrow B$ by $\varphi(\pi(a)):=\phi(a)$, for $a\in A$. Clearly $\varphi$ is a $\ast$-homomorphism satisfying $\varphi\circ\pi=\phi$. Also, since $\pi$ is onto, a simple diagram chase shows that $f\circ\varphi=g$.

Assume, moreover, that $g$ is faithful. For $d$ in $D$,
\begin{align*}\varphi(d)=0&\Rightarrow \varphi(d^*)\varphi(d)=0\Rightarrow f(\varphi(d)^*\varphi(d))=0\Rightarrow f\varphi(d^*d)=0\\&\Rightarrow g(d^*d)=0\Rightarrow d=0.
\end{align*}
\end{proof}

\begin{proposition}\label{prop: Fell absorption}
Let $\pi:=\{\pi_t:B_t\to\BH\}_{t\in G}$ be a representation of a Fell bundle in $\BH$ for some Hilbert space $\cH$. There is a $\ast$-homomorphism
\[\varphi_\pi:\cstarl{\cB}\too\bB(\cH\otimes_2\ell_2(G));\quad \varphi_\pi(a\delta_s)=\pi_s(a)\otimes\lambda_s^G.\]
If $\pi$ is faithful, so is $\varphi_\pi$.
\end{proposition}

\begin{proof}
As defined in~\eqref{eq: pi tensor with lambda}, we may tensor the representation $\pi$ with the left-regular representation of $G$ to yield the representation $\pi\otimes\lambda^G:\cB\to\bB(\cH\otimes_2\ell_2(G))$, and by universality obtain the $\ast$-homomorphism
\[\pi\otimes\lambda^G:\cstar(\cB)\to\bB(\cH\otimes_2\ell_2(G));\quad b\delta_t\mapsto\pi_t(b)\otimes\lambda_t^G.\]

Next, for $t\in G$ consider the closed subspaces $M_t:=\cH\otimes_2\Span(\epsilon_t)$ of $\cH\otimes_2\ell_2(G)$ along with their respective orthogonal projections $P_t=I\otimes e_{t,t}$. Clearly we have a Hilbert internal direct sum
\begin{equation}\label{eq: block diag decomp}
\cH\otimes_2\ell_2(G)=\bigoplus_{t\in G}M_t.
\end{equation}
It is well-known that
\[\bE_D:\bB(\cH\otimes_2\ell_2(G))\too\bB(\cH\otimes_2\ell_2(G));\quad \bE_D(x)=\sum_{t\in G}P_txP_t\quad\text{(SOT)}\]
is a faithful conditional expectation onto $D$; the \cstar-subalgebra of $\bB(\cH\otimes_2\ell_2(G))$ consisting of all operators which are block-diagonal with respect to the decomposition~\eqref{eq: block diag decomp}.

Also, it is fairly clear that \[\phi:\cstarl{\cB}\to\bB(\cH\otimes_2\ell_2(G));\quad\phi(x)=\pi_e(\bE(x))\otimes I\]
is linear and positive map. Moreover, we claim that $\bE_D\circ\pi\otimes\lambda^G=\phi\circ\pi_\lambda$, where $\pi_\lambda:\cstar(\cB)\to\cstarl{\cB}$ is the canonical quotient mapping. 
\[\begin{tikzcd}
\cstar(\cB) \arrow[r, "\pi\otimes\lambda^G"]	\arrow[rd, "\pi_\lambda"']	&\bB(\cH\otimes_2\ell_2(G)) \arrow[r, "\mathbb{E}_D"]&	\bB(\cH\otimes_2\ell_2(G))	\\
&	\cstar_{\lambda}(\cB) \arrow[red,u,dashed, "\exists\varphi_\pi"]	\arrow[ru, "\phi"']	&
\end{tikzcd}
\]
Indeed, for $x=\sum_{s\in G}b_s\delta_s$ in $\rC_c(\cB)$,
\begin{align*}
\bE_D\circ(\pi\otimes\lambda^G)(x)&=\bE_D\circ(\pi\otimes\lambda^G)\bigg(\sum_{s\in G}b_s\delta_s\bigg)=\bE_D\bigg(\sum_{s\in G}\pi_s(b_s)\otimes\lambda_s^G\bigg)=\pi_e(b_e)\otimes I\\&=\pi_e(\bE(x))\otimes I=\phi(x)=\phi\circ\pi_\lambda(x).
\end{align*}
Our claim follows by continuity.

By lemma~\ref{lem: factorlemma} there is now a $\ast$-homomorphism $\varphi_\pi$ as in the above diagram making the diagram commute. Note that if $\pi$ is faithful, so is $\phi$, so the lemma ensures that $\varphi_\pi$ is faithful as well.
\end{proof}

\begin{corollary}\label{cor: Fell cor}
Let $\pi:=\{\pi_t:B_t\to M\}_{t\in G}$ be a representation of a Fell bundle $\cB$ in a \cstar-algebra $M$. There is a $\ast$-homomorphism
\[\varphi:\cstarl{\cB}\to M\otimes\cstarl{G};\quad a\delta_s\mapsto\pi_s(a)\otimes\lambda_s^G.\]
If $\pi$ is faithful, so is $\varphi$.
\end{corollary}

\begin{proof}
We may consider $M\subseteq\BH$ for some Hilbert space $\cH$. Apply Proposition~\ref{prop: Fell absorption} and obtain $\varphi_\pi:\cstarl{B}\to\bB(\cH\otimes_2\ell_2(G))$ with $\varphi_\pi\circ\pi_\lambda=\pi\otimes\lambda^G$. Since $\varphi_\pi(a\delta_s)=\pi_s(a)\otimes\lambda_s^G$, we see that  $\varphi_\pi(\rC_c(\cB))\subseteq M\otimes\cstarl{G}$, whence $\Ran(\varphi_\pi)\subseteq M\otimes\cstarl{G}$.
\end{proof}

\section{Residually Finite Partial Actions}\label{sec: RF actions}

The notion of a residually finite action was first introduced By Kerr and Nowak in~\cite{KerrNowak}. Here we generalize this idea to partial dynamical systems and establish similar results. We will also define residually finite-dimensional partial actions generalizing a similar notion from~\cite{Rain2014}.

\begin{definition}\label{def: RF action}
	Let $\theta:=\big\{\theta_t:U_{t^{-1}}\to U_t\big\}_{t\in G}$ be a continuous partial action of $G$ on a metric space $X$. We will call $\theta$ \emph{residually finite} (RF) if for every $\delta>0$ and every finite subset $F\subseteq G$, there are a finite set $Z$, a partial action of $G$ on $Z$; $\eta:=\big\{\eta_t:V_{t^{-1}}\to V_t\big\}_{t\in G}$, and a map $\rho:Z\to X$ satisfying
	\begin{enumerate}[(i)]
		\item $\rho(V_t)\subseteq U_t$, and $\rho^{-1}(U_t)\subseteq V_t$ for all $t\in F$,
		\item $d(\rho(\eta_t(z)),\theta_t(\eta(z)))<\delta$ for all $t\in F$ and $z\in V_{t^{-1}}$,
		\item $X\subseteq_\delta\rho(Z)$, that is, the range of $\rho$ is $\delta$-dense in $X$.
		
	\end{enumerate}
\end{definition}

As with global residually finite actions, residually finite partial actions admit invariant measures. Recall that if $\theta:=\big\{\theta_t:U_{t^{-1}}\to U_t\big\}_{t\in G}$ is a continuous partial action of $G$ on a topological space $X$, a Borel measure $\mu$ on $X$ is $G$-invariant if for every open $V\subseteq X$ and every $t\in G$ we have $\mu(\theta_{t}(V\cap U_{t^{-1}}))=\mu(V\cap U_{t^{-1}})$.

\begin{proposition}
Let $\theta:=\big\{\theta_t:U_{t^{-1}}\to U_t\big\}_{t\in G}$ be a residually finite continuous partial action of $G$ on a compact metric space $X$, and let $\alpha:=\big\{\alpha_t:\rC_0(U_{t^{-1}})\to\rC_0(U_t)\big\}_{t\in G}$ be the associated partial \cstar-system. Then $\rC(X)$ admits an $G$-invariant (tracial) state. Consequently, $X$ admits an invariant Borel probability measure. 
\end{proposition}

\begin{proof}
	Consider the directed set $I=\big\{(F,\delta)\mid F\subseteq G \text{ finite subset, } \delta>0\big\}$, where $(F,\delta)\geq(F',\delta')$ if $F\supseteq F'$ and $\delta<\delta'$.
	Fix $(F,\delta)$ in $I$. Definition~\ref{def: RF action} then yields a finite set $Z$, a partial action of $G$ on $Z$; $\eta:=\big\{\eta_t:V_{t^{-1}}\to V_t\big\}_{t\in G}$, and a map $\rho:Z\to X$ satisfying the conditions of an RF action. Consider the state
	\[\psi:\rC(Z)\to\bC;\quad\psi(g)=\frac{1}{|Z|}\sum_{z\in Z}g(z),\]
	composed with the unital $\ast$-homomorphism $\overline{\rho}:\rC(X)\to\rC(Z)$. We thus have a state
	\[\mu_{(F,\delta)}:\rC(X)\to\bC;\quad \mu_{(F,\delta)}(f)=\frac{1}{|Z|}\sum_{z\in Z}f(\rho(z)),\]
	and a net of states $\big(\mu_{(F,\delta)}\big)_{(F,\delta)}$. We claim that for each $s\in G$ and $f\in\rC_0(U_{s^{-1}})$
	\[\big|\mu_{(F,\delta)}(\alpha_s(f))-\mu_{(F,\delta)}(f)\big|\stackrel{(F,\delta)\to\infty}{\too}0.\]
	Given $\varepsilon>0$, find $\delta_0>0$ such that
	\[x,y\in X,\quad d(x,y)<\delta_0\implies |f(x)-f(y)|<\varepsilon.\]
	Now set $F_0=\{s^{-1},e,s\}$ and fix $(F,\delta)\geq(F_0,\delta_0)$. With $Z$, $\rho$, and $\eta$ associated to the pair $(F,\delta)$ we have  $d(\theta_{s^{-1}}(\rho(z)),\rho(\eta_{s^{-1}}(z)))<\delta<\delta_0$ for each $z\in V_s$. So 
	\begin{align*}
	\big|\mu_{(F,\delta)}(\alpha_s(f))-\mu_{(F,\delta)}(f)\big|&=\bigg|\frac{1}{|Z|}\sum_{z\in Z}\alpha_s(f)(\rho(z))-\frac{1}{|Z|}\sum_{z\in Z}f(\rho(z))\bigg|\\
	&=\bigg|\frac{1}{|Z|}\sum_{z\in V_s}f(\theta_{s^{-1}}(\rho(z)))-\frac{1}{|Z|}\sum_{z\in V_{s^{-1}}}f(\rho(z))\bigg|\\
	&\leq\frac{1}{|Z|}\sum_{z\in V_s}\big|f(\theta_{s^{-1}}(\rho(z)))-f(\rho(\eta_{s^{-1}}(z)))\big|<\varepsilon.
	\end{align*}
	This proves the claim. 
	
	Now let $\mu$ be a weak*-cluster point of the net $\big(\mu_{(F,\delta)}\big)_{(F,\delta)}$. Using a subnet and an $\varepsilon/3$-argument we get that $\mu(\alpha_s(f))=\mu(f)$ for every $s\in G$ and $f\in \rC_0(U_{s^{-1}})$. Thus $\mu$ is a $G$-invariant state which gives a $G$-invariant Borel probability measure on $X$.
\end{proof}

The notion of an residually finite-dimensional action in the global setting was introduced in~\cite{Rain2014} (Definition 3.2). This notion is more restrictive than RF actions as we require precise equivariance rather than approximate equivariance. Here we adopt that idea to partial systems. 

\begin{definition}\label{def: RFD action}
	Let $\theta:=\big\{\theta_t:U_{t^{-1}}\to U_t\big\}_{t\in G}$ be a continuous partial action of $G$ on a metric space $X$. We will call $\theta$ \emph{residually finite-dimensional} (RFD) if for every $\delta>0$ there is a finite set $Z$, a partial action of $G$ on $Z$; $\eta:=\big\{\eta_t:V_{t^{-1}}\to V_t\big\}_{t\in G}$, and a map $\rho:Z\to X$ satisfying
	\begin{enumerate}[(i)]
		\item $\rho(V_t)\subseteq U_t$, and $\rho^{-1}(U_t)\subseteq V_t$ for all $t\in G$,
		\item $\rho(\eta_t(z))=\theta_t(\eta(z))$ for all $t\in G$ and $z\in V_{t^{-1}}$.
		\item $X\subseteq_\delta\rho(Z)$, that is, the range of $\rho$ is $\delta$-dense in $X$,
	\end{enumerate}
\end{definition}

Note that (i) and (ii) say that $\rho$ is strictly equivariant. We will show below that RFD actions characterize RFD partial crossed products. We end this section looking at the partial Bernoulli action.

\begin{example}\label{ex: Bernoulli shift}
	Let $G$ be a group, and set $\Pi:=\{0,1\}^G$ equipped with the product topology. If $G$ is countable, say $G=\{e=t_0, t_1, t_2,\dots\}$, then $\Pi$ is metrizable via
	\[d(x,y)=\sum_{k=1}^\infty2^{-k}|y(t_k)-x(t_k)|.\]
	For each $t\in G$, let $\pi_t:\Pi\to\{0,1\}$ denote the canonical projection: $\pi_t(x)=x(t)$ and set
	\[X_t:=\pi_t^{-1}(\{1\})=\big\{x\in\Pi\mid x(t)=1\big\}.\]
	Note that each $X_t\subseteq\Pi$ is clopen. Write $X:=X_e$.
	
	Let $\sigma:G\to\Sym(G)$ denote the canonical Cayley action. Recall that the Bernoulli action is defined as
	\[\beta^G:G\to\Homeo(\Pi);\quad t\mapsto\beta_t^G,\quad \beta_t^G(x)=x\circ\sigma_{t^{-1}}.\]
	We will suppress the decorative `$G$' when the group is understood and simply write $\beta$.
	Note that $\beta_s(X_t)=X_{st}$ for all $s,t\in G$. For each $t\in G$ we define the clopen subset of $X$:
	\[U_t:=X\cap\beta_t(X)=X_e\cap X_t=\big\{x\in\Pi\mid x(e)=x(t)=1\big\}\subseteq X.\]
	Note that
	\[\beta_t(U_{t^{-1}})=\beta_t(X_e\cap X_{t^{-1}})=\beta_t(X_e)\cap\beta_t(X_{t^{-1}})=X_t\cap X_e=U_t.\]
	We now have a continuous partial action \begin{equation}\label{eq: Bernoulli action}
	\theta:G\to\pHomeo(X);\quad \{\theta_t:U_{t^{-1}}\to U_t\}_{t\in G};\quad \theta_t(x)=\beta_t(x).
	\end{equation}
	This action is known as the \emph{Bernoulli partial action} of $G$. We will sometimes denote this partial action succinctly as $\theta^G:G\to\pHomeo(X_G)$.
	
	Now suppose $\Gamma$ is a group and $\phi:G\to\Gamma$ is a homomorphism.
	We set $\Sigma:=\{0,1\}^\Gamma$, again with the product topology. There is now a continuous action of $G$ on $\Sigma$:
	\[\mu:G\to\Homeo(\Sigma);\quad t\mapsto\mu_t;\quad\mu_t(z)=z\circ\sigma_{\phi(t)^{-1}}.\]
	Indeed, this action is just the composition $G\stackrel{\phi}{\too}\Gamma\stackrel{\beta^\Gamma}{\too}\Homeo(\Sigma)$. As above we consider the sets
	\[Z_t:=\big\{z\in\Sigma\mid z(\phi(t))=1\big\};\]
	clopen subsets of $\Sigma$. Write $Z:=Z_e$. Again we see that $\mu_s(Z_t)=Z_{st}$ for $s,t\in G$. For each $t\in G$ we define the clopen subset of $Z$:
	\[V_t=Z\cap\mu_t(Z)=Z_e\cap Z_t=\big\{z\in\Sigma\mid z(e)=z(\phi(t))=1\big\}\subseteq Z.\]
	We again see that $\mu_t(V_{t^{-1}})=V_t$, so
	we obtain a continuous partial action
	\[\eta:G\to\pHomeo(Z);\quad \{\eta_t:V_{t^{-1}}\to V_t\}_{t\in G};\quad \eta_t(x)=\mu_t(x).\]
	Note that the map
	\[\rho:Z\to X;\quad\rho(z)=z\circ\phi\]
	is well-defined and strictly  $G$-equivariant. Indeed,
	$\rho(z)(e)=z\circ\phi(e)=z(e)=1$, so $\rho(z)\in X$. If $z\in V_t$, then $\rho(z)(t)=z(\phi(t))=1$, so $\rho(z)\in U_t$. Thus $\rho(V_t)\subseteq U_t$. Now if $z\in V_{t^{-1}}$,
	\[\rho(\eta_t(z))=\eta_t(z)\circ\phi=z\circ\sigma_{\phi(t)^{-1}}\circ\phi;\quad s\mapsto z(\phi(t)^{-1}\phi(s))=z(\phi(t^{-1}s)),\]
	whereas
	\[\theta_t(\rho(z))=\rho(z)\circ\sigma_{t^{-1}}=z\circ\phi\circ\sigma_{t^{-1}};\quad s\mapsto z(\phi(t^{-1}s)),\]
	so $\rho(\eta_t(z))=\theta_t(\rho(z))$.	This gives the equivariance. Moreover, we also have $\rho^{-1}(U_t)\subseteq V_t$. For if $z\in\rho^{-1}(U_t)$, then $\rho(z)\in U_t$, so 
	\[z(\phi(t))=z\circ\phi(t)=\rho(z)(t)=1,\]
	which implies $z\in V_t$.

	\begin{claim}
		If $G$ is a countable and residually finite group, then the partial Bernoulli action~\eqref{eq: Bernoulli action} is RFD.
	\end{claim}
	\begin{proof}
		Let $\delta>0$ and enumerate $G=\big\{e=t_0,t_1,t_2,\dots\big\}$. Choose $N$ so large so that
		\[\sum_{k>N}2^{-k}<\delta.\]
		Since $G$ is residually finite there is a finite group $\Gamma$, and a homomorphism $\phi:G\to\Gamma$ with
		\[\phi(t_k)\neq e\quad\forall k=1,\dots,N, \text{ and }\ \phi(t_i)\neq\phi(t_j)\quad \forall i\neq j\in\{0,\dots, N\}.\]		
		
		Now consider $\Sigma$, $Z$, $\mu$, and $\eta$ as in our above discussion. We need only verify (ii) of Definition~\ref{def: RFD action}. To that end, given $x\in X$, define $z_x\in\Sigma$ as follows:
		\[z_x(\gamma)=\left\{
		\begin{array}{ll}
		x(t_k) & \quad \gamma=\phi(t_k), \text{ for } k=0,\dots,N, \\
		0 & \quad \gamma\notin\phi(\{t_0,\dots, t_N\})
		\end{array}
		\right.\]
		Note that $z_x(e)=x(e)=1$ since $t_0=e$, and $\phi(t_0)=\phi(e)=e$. Thus $z_x\in Z$. Now for $k=0,1,\dots,N$ we have
		\[\rho(z_x)(t_k)=z_x(\phi(t_k))=x(t_k),\]
		so
		\[d(\rho(z_x),x)=\sum_{k=0}^\infty2^{-k}|\rho(z_x)(t_k)-x(t_k)|=\sum_{k>N}2^{-k}|\rho(z_x)(t_k)-x(t_k)|\leq\sum_{k>N}2^{-k}<\delta.\]
		Thus $\theta$ is RFD.
	\end{proof}

\end{example}

\section{Approximation of Fell Bundles and their \cstar-algebras}\label{sec: approx of Fell bundles}

In this section we characterize  residual-finite-dimensionality, quasidiagonality, and the MF property in reduced cross-sectional \cstar-algebras constructed from Fell bundles over discrete groups. As these properties are all expressed in terms of external approximating maps into matrix algebras, it seems natural that the analogous properties for Fell bundles should have the flavor of approximate representations into such algebras.

\subsection{Residually Finite-Dimensional Bundles}
We begin with the residual finite-dimensional property. Recall that a \cstar-algebra $A$ is \emph{residually finite-dimensional} (RFD) if for every $\varepsilon>0$ and every finite subset $\Omega\subseteq A$, there is a $d\in\bN$ and a $\ast$-homomorphism $\varphi:A\to\bM_d$ with
\[\|\varphi(a)\|\geq\|a\|-\varepsilon,\quad\forall a\in \Omega.\]
In fact, it suffices to consider finite subsets $\Omega\subseteq A_0$, where $A_0\subseteq A$ is a dense $\ast$-subalgebra of $A$.
If $A$ is separable, $A$ is RFD if and only if there is a natural sequence $(k_n)_{n\geq1}$ and a sequence of $\ast$-homomorphisms $(\varphi_n:A\to\bM_{k_n})_{n\geq1}$ satisfying:
\begin{equation}\label{eqn: RFD}
\|\varphi_n(a)\|\stackrel{n\to\infty}{\too}\|a\|\quad\forall a\in A.
\end{equation}
Again,~\eqref{eqn: RFD} need only hold for all $a$ in some dense $\ast$-subalgebra $A_0\subseteq A$. Equivalently, a separable $A$ is RFD if and only if there is an embedding
\[A\hookrightarrow\prod_{n\geq1}\bM_{k_n}\]
for some natural sequence $(k_n)_{n\geq1}$.

Here is the relevant RFD property for Fell bundles.

\begin{definition}\label{def: RFD bundle}
A Fell bundle $\cB=\{B_t\}_{t\in G}$
is said to be \emph{Residually Finite-Dimensional} (RFD) if for every $\varepsilon>0$ and every finite subset of the total space $\Omega\subseteq B$, there is a $d\in\bN$ and a representation $\pi:\cB\to\bM_d$ with
\[\|\pi_{t}(b)\|\geq\|b\|-\varepsilon,\quad\forall b\in \Omega\cap B_t.\]
\end{definition}

It seems more natural to allow finite subsets of the total space, but note that it is enough to consider finite subsets of the unit fiber algebra $B_e$. Indeed, if $\Omega\subseteq B$ is a finite subset of the total space, then $\Omega'=\big\{b^*b\mid b\in\Omega\big\}\subseteq B_e$ is finite, and if a representation $\pi:\cB\to\bM_d$ is approximately isometric within $\varepsilon$ on $\Omega'$, then
\[\|\pi_t(b)\|^2=\|\pi_t(b)^*\pi_t(b)\|=\|\pi_{t^{-1}}(b^*)\pi_t(b)\|=\|\pi_e(b^*b)\|\geq\|b^*b\|-\varepsilon=\|b\|^2-\varepsilon.\]

In the separable case we see that $\cB$ is RFD if and only if there is a natural sequence $(k_n)_{n\geq1}$ and a sequence of representations $(\pi_n:\cB\to\bM_{k_n})_{n\geq1}$ satisfying
\[\|\pi_{n,t}(b)\|\stackrel{n\to\infty}{\too}\|b\|\quad \forall t\in G,\ b\in B_t.\]

\begin{proposition}\label{prop: RFD bundle}
Let $\cB=\{B_t\}_{t\in G}$ be a Fell bundle.
\begin{enumerate}[(1)]
	\item If $\cstarl{\cB}$ is RFD, then $\cB$ is RFD.
	\item Assume $\cB$ is separable. If $\cB$ and $\cstarl{G}$ are RFD, then so is $\cstarl{\cB}$.
\end{enumerate}
\end{proposition}

\begin{proof}
(1): Let $\Omega\subseteq B$ be a finite subset of the total space and suppose $\varepsilon>0$. Now set
\[\Sigma:=\bigcup_{t\in G}\big\{b\delta_t\mid b\in\Omega\cap B_t\big\}\subseteq\rC_c(\cB).\]
Note that $\Sigma$ is finite. Since $\cstarl{\cB}$ is RFD there is a $d\in\bN$ and a $\ast$-homomorphism $\phi:\cstarl{\cB}\to\bM_d$ approximately isometric on $\Sigma$ within $\varepsilon$. Consider the composed representation
\[\pi:=\phi\circ j:=\big\{\phi\circ j_t:B_t\to\bM_d\big\}_{t\in G}, \]
where $j:\cB\to\rC_c(\cB)\subseteq\cstarl{\cB}$ is the canonical representation given by $j_t(b)=b\delta_t$. Then if $b\in\Omega\cap B_t$ we have
\[\|\pi_t(b)\|=\|\phi(j_t(b))\|\geq\|b\delta_t\|-\varepsilon=\|b\|-\varepsilon.\]
Thus $\cB$ is RFD.

(2): Assume $(\pi_n:\cB\to\bM_{k_n})_{n\geq1}$ is an approximating sequence of representations as in the remark proceeding Definition~\ref{def: RFD bundle}. As in~\eqref{eq: direct sum rep} we may take the direct sum representation
\[\pi:=\oplus_{n\geq1}\pi_n:\cB\too M:=\prod_{n\geq1}\bM_{k_n};\quad \pi_t(b)=(\pi_{n,t}(b))_{n\geq1}.\]
Note that $\pi$ is faithful since
\[\|\pi_e(a)\|=\|(\pi_{n,e}(a))_{n}\|=\sup_{n\geq1}\|\pi_{n,e}(a)\|=\|a\|.\]
By Corollary~\ref{cor: Fell cor} we have an injective $\ast$-homomorphism
\[\varphi_\pi:\cstarl{\cB}\too M\otimes\cstarl{G}.\]
Since $M$ and $\cstarl{G}$ are RFD, so is their minimal tensor product. Thus $\cstarl{\cB}$ is RFD.
\end{proof}

Restricting our attention to partial actions we hope to describe the RFD property dynamically. The following is a generalization of an RFD global action defined in~\cite{Rain2014}. 

\begin{definition}
Let $\alpha:G\to\pAut(A)$ be a partial \cstar-dynamical system with partial automorphisms $\big\{\alpha_t:D_{t^{-1}}\to D_t\big\}_{t\in G}$. We say that $\alpha$ is \emph{residually finite-dimensional} (RFD) if for every $\varepsilon>0$ and every finite subset $\Omega\subseteq A$ there is a $d\in\bN$ and covariant representation
\[(\varphi,v):(A,G,\alpha)\too\bM_d\]
with
\[\|\varphi(a)\|\geq \|a\|-\varepsilon,\quad\forall a\in\Omega.\]
\end{definition}

If $A$ is separable and $G$ is countable, we can see that $\alpha:G\to\pAut(A)$ is RFD if and only if there is a natural sequence $(k_n)_{n\geq1}$ and a sequence of covariant representations $(\varphi_n,v_n):(A,G,\alpha)\too\bM_{k_n}$ with
\[\|\varphi_n(a)\|\stackrel{n\to\infty}{\too}\|a\|\quad\forall a\in A.\]

\begin{proposition}\label{prop: RFD crossed product}
Let $\alpha:G\to\pAut(A)$ be a partial \cstar-dynamical system with partial automorphisms $\big\{\alpha_t:D_{t^{-1}}\to D_t\big\}_{t\in G}$, and write $\cA_\alpha=\{A_t\}_{t\in G}$ for the associated Fell bundle. 
\begin{enumerate}
	\item The action $\alpha$ is RFD if and only if the bundle $\cA_\alpha$ is RFD.
	\item If the partial reduced crossed product $A\rtimes_{\lambda}^\alpha G$ is RFD, so is the action $\alpha$.
	\item If $A$ is separable, $G$ is countable, and if $\alpha$ and $\cstarl{G}$ are RFD, then the reduced crossed product $A\rtimes_{\lambda}^\alpha G$ is RFD.
\end{enumerate}
\end{proposition}

\begin{proof}
(1): We first assume the action $\alpha$ is RFD. Let $\Omega\subseteq A_e=A$ be a finite subset and suppose $\varepsilon>0$. Find a $d\in\bN$ and a covariant representation $(\varphi,v)$ of $(A,G,\alpha)$ in $\bM_d$ with $\varphi$ almost isometric on $\Omega$. We then obtain a $\ast$-homomorphism
\[\varphi\rtimes v:A\rtimes_{\text{alg}}^\alpha G\to\bM_d;\quad (\varphi\rtimes v)(a\delta_s)=\varphi(a)v_s.\]
Since $\rC_c(\cA_\alpha)\cong A\rtimes_{\text{alg}}^\alpha G$ as $\ast$-algebras, we thus get a $\ast$-homomorphism $\rC_c(\cA_\alpha)\to\bM_d$ which in turn yields a representation of $\cA_\alpha$ in $\bM_d$:
\[\{\pi_s:A_s\to\bM_d\}_{s\in G};\quad\pi_s(a)=\varphi(a)v_s.\]
Given $a\in\Omega$,
\[\|\pi_e(a)\|=\|\varphi(a)v_e\|=\|\varphi(a)\|\geq\|a\|-\varepsilon,\]
so $\cA_\alpha$ is RFD.

Now suppose the bundle $\cA_\alpha$ is RFD. Given a finite subset $\Omega\subseteq A$ and $\varepsilon>0$, there is a $d\in\bN$ and a representation $\pi:\cA_\alpha\to\bM_d$ with
\[\|\pi_e(a)\|\geq\|a\|-\varepsilon\quad\forall a\in\Omega.\]
We then get a $\ast$-homomorphism $\pi:\rC_c(\cA_\alpha)\to\bM_d$ with $\pi((a,s)\delta_s)=\pi_s(a)$, which induces a $\ast$-homomorphism \[\rho:A\rtimes_{\text{alg}}^\alpha G\to\bM_d;\quad\rho(a\delta_s)=\pi_s(a)\quad \forall s\in G,\ a\in D_s.\]
By the universal property $\rho$ extends to a $\ast$-homomorphism $\rho:A\rtimes^\alpha G\to\bM_d$. Using a result in~\cite{Mc95} there is a covariant representation $(\varphi,v)$ of $(A,G,\alpha)$ in $\bM_d$ with $\varphi\rtimes v=\rho$. Now if $a\in\Omega$
\[\|\varphi(a)\|=\|\rho(a\delta_e)\|=\|\pi_e(a)\|\geq\|a\|-\varepsilon.\]
Thus $\alpha$ is RFD.

(2) and (3) follow from (1), Proposition~\ref{prop: RFD bundle}, and the fact that $\cstarl{\cA_\alpha}=A\rtimes_{\lambda}^\alpha G$.
\end{proof}

Our next goal is to show that an
RFD continuous partial action $\theta:G\to\pHomeo(X)$ gives rise to a RFD  partial action at the \cstar-algebraic level $\alpha:G\to\pAut(\rC(X))$ and that the converse holds as well (see Proposition~\ref{prop: cts RFD implies RFD}). To accomplish this we shall need to unpack the structure of partial actions on finite sets and partial representations into matrix algebras. This is the content of the following two propositions.

\begin{proposition}\label{prop: finite partial systems}
Let $\eta=\big\{\eta_t:V_{t^{-1}}\to V_t\big\}_{t\in G}$ be a partial action of $G$ on a finite set $Z$, and let $\beta=\big\{\beta_t:\rC_0(V_{t^{-1}})\to\rC_0(V_t))\big\}_{t\in G}$ be the dual partial action of $G$ on $\rC(Z)$.

There is a $d\in\bN$, and a faithful covariant representation 
\[(\phi,v):(\rC(Z),G,\beta)\too\bM_d.\]
\end{proposition}

\begin{proof}
Let $d:=|Z|$. For every $V\subseteq Z$ we see that $\rC_0(V)=\Span(\{\delta_z\mid z\in V\})$, where $\delta_z(x)=\delta_{z,x}$. Note that
\[\beta_t(\delta_z)=\delta_{\eta_t(z)}\quad\forall t\in G,\ z\in V_{t^{-1}}.\]
Also, write $\{\epsilon_z\}_{z\in Z}$ for the canonical orthonormal basis in $\ell_2(Z)$, and for each $x,y$ in $Z$, $\e_{x,y}:=\epsilon_x\otimes\overline{\epsilon_y}$ will denote the standard matrix unit. Moreover, for each $t\in G$ set
\[H_t:=\Span\big(\big\{\epsilon_t\mid z\in V_{t}\big\}\big)\subseteq\ell_2(Z),\]
and
\[v_t=\sum_{z\in V_{t^{-1}}}\e_{\eta_t(z),z}\in \bB(\ell_2(Z))\cong\bM_{d}.\]
Then $v_t$ is a partial isometry; indeed,
\[v_t^*v_t=\bigg(\sum_{z\in V_{t^{-1}}}\e_{\eta_t(z),z}\bigg)^*\bigg(\sum_{z\in V_{t^{-1}}}\e_{\eta_t(z),z}\bigg)=\sum_{z,y\in V_{t^{-1}}}\e_{y,\eta_t(y)}\e_{\eta_t(z),z}=\sum_{z\in V_{t^{-1}}}\e_{z,z},\]
is the orthogonal projection onto $H_{t^{-1}}$, and similarly $v_tv_t^*$ is the orthogonal projection onto $H_t$.
\begin{claim}
$v:G\to\bM_{d}$ is a partial $\ast$-representation.
\end{claim}

We verify the conditions described in~\ref{def: partial rep}. Indeed,
\[v_e=\sum_{z\in Z}\e_{\eta_e(z),z}=\sum_{z\in Z}\e_{z,z}=1_d,\]
and
\[v_{t}^*=\bigg(\sum_{z\in V_{t^{-1}}}\e_{\eta_t(z),z}\bigg)^*=\sum_{z\in V_{t^{-1}}}\e_{z,\eta_t(z)}\stackrel{y=\eta_t(z)}{=}\sum_{y\in V_{t}}\e_{\eta_{t^{-1}}(y),y}=v_{t^{-1}}.\]
Next,
\begin{align*}
v_s^*v_sv_t&=(v_s^*v_s)v_t=\bigg(\sum_{z\in V_{s^{-1}}}\e_{z,z}\bigg)\bigg(\sum_{y\in V_{t^{-1}}}\e_{\eta_t(y),y}\bigg)=\sum_{z\in V_{s^{-1}},\ y\in V_{t^{-1}}}\e_{z,z}\e_{\eta_t(y),y}\\
&=\sum_{z\in V_{s^{-1}},\ y\in V_{t^{-1}}}\ip{\epsilon_{\eta_t(y)}}{\epsilon_z}\e_{z,y}=\sum_{z\in V_{s^{-1}},\ y\in \eta_t^{-1}(V_{s^{-1}})}\ip{\epsilon_{\eta_t(y)}}{\epsilon_z}\e_{z,y}
\\&=\sum_{y\in \eta_t^{-1}(V_{s^{-1}})}\e_{\eta_t(y),y}
\end{align*}
On the other hand,
\begin{align*}
v_s^*v_{st}&=\bigg(\sum_{z\in V_{s^{-1}}}\mathrm{e}_{\eta_s(z),z}\bigg)^*\sum_{y\in V_{(st)^{-1}}}\e_{\eta_{st}(y),y}=\sum_{z\in V_{s^{-1}},\ y\in V_{(st)^{-1}}}\e_{z,\eta_s(z)}\e_{\eta_{st}(y),y}\\
&=\sum_{z\in V_{s^{-1}},\ y\in V_{(st)^{-1}}}\ip{\epsilon_{\eta_{st}(y)}}{\epsilon_{\eta_s(z)}}\e_{z,y}
\end{align*}
Now if $y\in V_{(st)^{-1}}$ and $z\in V_{s^{-1}}$ with $\eta_{st}(y)=\eta_s(z)$ in $V_s\cap V_{st}$, we may apply $\eta_{s^{-1}}$ to both sides and get
\[z=\eta_{s^{-1}}\eta_s(z)=\eta_{s^{-1}}\eta_{st}(y)=\eta_{s^{-1}}\eta_s\eta_t(y)=\eta_t(y),\]
and deduce that $y\in \eta_t^{-1}(V_{s^{-1}})$. We thus arrive at
\[v_s^*v_{st}=\sum_{z\in V_{s^{-1}},\ y\in V_{(st)^{-1}}}\ip{\epsilon_{\eta_{st}(y)}}{\epsilon_{\eta_s(z)}}\e_{z,y}=\sum_{y\in \eta_t^{-1}(V_{s^{-1}})}\e_{\eta_t(y),y}.\] 
Therefore $v_s^*v_sv_t=v_s^*v_{st}$, and $v$ is a partial $\ast$-representation as claimed.

Now consider the $\ast$-monomorphism
\[\phi:\rC(Z)\to\bM_d;\quad \phi(\delta_z)=\e_{z,z}.\]
We verify that $(\phi,v)$ is covariant. To this end let $t\in G$ and $z\in V_{t^{-1}}$.
\begin{align*}v_t\phi(\delta_z)v_t^*&=\bigg(\sum_{x\in V_{t^{-1}}}\e_{\eta_t(x),x}\bigg)\e_{z,z}\bigg(\sum_{y\in V_{t^{-1}}}\e_{y,\eta_t(y)}\bigg)
=\sum_{x\in V_{t^{-1}},\ y\in V_{t^{-1}}}\ip{\epsilon_z}{\epsilon_x}\e_{\eta_t(x),z}\e_{y,\eta_t(y)}\\
&=\sum_{y\in V_{t^{-1}}}\e_{\eta_t(z),z}\e_{y,\eta_t(y)}=
\sum_{y\in V_{t^{-1}}}\ip{\epsilon_y}{\epsilon_z}\e_{\eta_t(z),\eta_t(y)}=\e_{\eta_t(z),\eta_t(z)}=\phi(\beta_t(\delta_z)).
\end{align*}
Now since every $f\in \rC_0(V_{t^{-1}})$ is a linear span of such $\delta_z$ with $z\in V_{t^{-1}}$, we have $v_t\phi(f)v_t^*=\phi(\beta_t(f))$ for every $f\in \rC_0(V_{t^{-1}})$, so $(\phi, v)$ is covariant.
\end{proof}

\begin{proposition}\label{prop: cov rep into mtx}
Let $\theta:=\big\{\theta_t:U_{t^{-1}}\to U_t\big\}_{t\in G}$ be a continuous partial action of $G$ on a compact space $X$, with its dual partial action $\alpha:=\big\{\alpha_t:\rC_0(U_{t^{-1}})\to \rC_0(U_t)\big\}_{t\in G}$. 

If $(\varphi,v):(\rC(X),G,\alpha)\to\bM_d$ is a covariant representation, then there is a finite set $Z$, a partial action of $G$ on $Z$; $\eta:=\big\{\eta_t:V_{t^{-1}}\to V_t\big\}_{t\in G}$, and a strictly equivariant map $\rho:Z\to X$ satisfying $\varphi(f)=f\circ\rho$ for all $f\in\rC(X)$. 
\end{proposition}

\begin{proof}
Since $\varphi(\rC(X))$ is a finite-dimensional commutative \cstar-algebra, we may identify it with $\rC(Z)$ for some finite set $Z$. We then get a map $\rho:Z\to X$ satisfying $\varphi(f)=f\circ\rho$. Also, for each $t\in G$ $\varphi(\rC_0(U_t))\subseteq\rC(E)$ is a closed ideal, whence $\varphi(\rC_0(U_t))=\rC_0(V_t)$ for some subset $V_t\subseteq Z$. Note that
\[\rC_0(V_s)\cap\rC_0(V_t)=\varphi(\rC_0(U_s))\cap\varphi(\rC_0(U_t))=\varphi(\rC_0(U_s)\cap\rC(U_t)).\footnote{If $\varphi:A\to B$ is a surjective $\ast$-homomorphism bewteen \cstar-algebras, and $I,J\subseteq A$ are closed ideals, then $\varphi(I),\varphi(J)\subseteq B$ are closed ideals and $\varphi(I\cap J)=\varphi(IJ)=\varphi(I)\varphi(J)=\varphi(I)\cap\varphi(J)$.}.\]

We claim that for every $t\in G$, $\rho(V_t)\subseteq U_t$ and $\rho^{-1}(U_t)\subseteq V_t$. Well, if $z\in V_t$, then there is an $f_z\in\rC_0(U_t)$ with $\varphi(f_z)=\delta_z$, so
\[1=\delta_z(z)=\varphi(f_z)(z)=f_z(\rho(z)),\]
which forces $\rho(z)\in U_t$. Also, if $z\in\rho^{-1}(U_t)$, then $\rho(z)\in U_t$. We can then find an $f\in\rC_0(U_t)$ with $f(\rho(z))=1$. We get $\varphi(f)\in\rC_0(V_t)$, but
\[\varphi(f)(z)=f(\rho(z))=1,\]
so $z\in V_t$.

Next, we claim that we have a partial action $\beta:G\to\pSym(Z)$;
\[\big\{\beta_t:\rC_0(V_{t^{-1}})\to\rC_0(V_t)\big\}_{t\in G};\quad \beta_t(g)=v_tgv_t^*.\]
Given $g\in\rC_0(V_{t^{-1}})$, there is an $f\in\rC_0(U_{t^{-1}})$ with $\varphi(f)=g$. So
\[v_tgv_t^*=v_t\varphi(f)v_t^*=\varphi(\alpha_t(f))\in\varphi(\rC_0(U_t))=\rC_0(V_t),\]
so each $\beta_t$ is well-defined and clearly $\ast$-linear. Similarly $\beta_t$ is a $\ast$-homomorphism. Now, given $g\in\rC_0(V_{t^{-1}})$ and $f\in\rC_0(U_{t^{-1}})$ with $\varphi(f)=g$, we have
\[\beta_{t^{-1}}\circ\beta_t(g)=v_{t^{-1}}v_tgv_t^*v_{t^{-1}}^*=v_{t^{-1}}v_{t^{-1}}^*gv_{t^{-1}}v_{t^{-1}}^*=\e_{t^{-1}}\varphi(f)\e_{t^{-1}}\stackrel{\eqref{eq: covariant rep relations}}{=}\varphi(f)=g.\]
Thus $\beta_t$ is a $\ast$-isomorphism. Now
\[\dom(\beta_s\circ\beta_t)=\beta_{t}^{-1}\big(\rC_{0}(V_t)\cap\rC_{0}(V_{s^{-1}})\big)=\beta_{t^{-1}}\big(\varphi(\rC_{0}(U_t)\cap\rC_{0}(U_{s^{-1}}))\big),\]
so if $f\in\rC_{0}(U_t)\cap\rC_{0}(U_{s^{-1}})$, then $\alpha_{t^{-1}}(f)$ is in $\rC_{0}(U_{t^{-1}})\cap\rC_{0}(U_{(st)^{-1}})$, so
\begin{align*}
\beta_{t^{-1}}(\varphi(f))&=v_{t^{-1}}\varphi(f)v_{t^{-1}}^*\stackrel{\eqref{eq: covariant rep relations}}{=}v_{t^{-1}}\e_{s^{-1}}\varphi(f)\e_{s^{-1}}v_{t^{-1}}^*
\stackrel{\eqref{eq: partial rep relations}}{=}\e_{t^{-1}s^{-1}}v_{t^{-1}}\varphi(f)v_{t^{-1}}^*\e_{t^{-1}s^{-1}}\\&=\e_{(st)^{-1}}\varphi(\alpha_{t^{-1}}(f))\e_{(st)^{-1}}\stackrel{\eqref{eq: covariant rep relations}}{=}\varphi(\alpha_{t^{-1}}(f))
\end{align*}
which belongs to $\varphi\big(\rC_{0}(U_{t^{-1}})\cap\rC_{0}(U_{(st)^{-1}})\big)=\rC_{0}(V_{t^{-1}})\cap\rC_{0}(V_{(st)^{-1}})\subseteq\rC_{0}(V_{(st)^{-1}})$. Thus $\dom(\beta_s\circ\beta_t)\subseteq\dom(\beta_{st})$. For such a $g=\varphi(\alpha_{t^{-1}}(f))$ in $\dom(\beta_s\circ\beta_t)$, 
\[\beta_s\circ\beta_t(g)=\beta_s\big(v_t\varphi(\alpha_{t^{-1}}(f))v_t^*\big)=\beta_s(\varphi(f))=v_s\varphi(f)v_s^*=\varphi(\alpha_s(f)),\]
whereas,
\[\beta_{st}(g)=v_{st}\varphi(\alpha_{t^{-1}}(f))v_{st}^*=\varphi\big(\alpha_{st}(\alpha_{t^{-1}}(f))\big)=\varphi(\alpha_s(f)).\]
Thus $\beta$ is a partial action on $\rC(Z)$. By duality we obtain a partial action $\eta:G\to\pSym(Z)$ with partial symmetries $\big\{\eta_t:V_{t^{-1}}\to V_t\big\}_{t\in G}$ satisfying
\[\beta_t(g)=g\circ\eta_{t^{-1}},\quad g\in\rC_0(V_{t^{-1}}).\]
We need only show that $\rho$ is equivariant. To this end let $z\in V_{t^{-1}}$ and $f\in\rC_0(U_t)$ be arbitrary. Then
\begin{align*}
f\big(\rho(\eta_t(z))\big)&=\varphi(f)(\eta_t(z))
=\beta_{t^{-1}}(\varphi(f))(z)=v_{t^{-1}}\varphi(f)v_{t^{-1}}^*=\varphi(\alpha_{t^{-1}}(f))(z)\\
&=\alpha_{t^{-1}}(f)(\rho(z))=f\big(\theta_t(\rho(z))\big).
\end{align*}
Since this holds for all $f\in\rC_0(U_t)$ we have $\rho(\eta_t(z))=\theta_t(\rho(z))$ as desired. 
\end{proof}

\begin{proposition}\label{prop: cts RFD implies RFD}
Let $\theta:G\to\pHomeo(X)$ be a continuous partial action of $G$ on a compact metric space $X$. Then $\theta$ is RFD if and only if the induced partial \cstar-action $\alpha:G\to\pAut(\rC(X))$ is RFD.
\end{proposition}

\begin{proof}
Let $\theta=\big\{\theta_t:U_{t^{-1}}\to U_t\big\}_{t\in G}$ be our partial homeomorphisms, and  let $\alpha=\big\{\alpha_t:D_{t^{-1}}\to D_t\big\}_{t\in G}$ the induced partial automorphisms of $\rC(X)$, where $D_t:=\rC_0(U_t)$.

First assume that $\theta$ is RFD. Let $\Omega\subseteq \rC(X)$ be a finite subset and $\varepsilon>0$. By uniform continuity there is a $\delta>0$ such that
\[x,y\in X,\ d(x,y)<\delta\implies |f(x)-f(y)|<\varepsilon\quad\forall f\in\Omega.\]
Run this $\varepsilon>0$ through Definition~\ref{def: RFD action} and obtain $Z$, $\eta:G\to\pSym(Z)$, and $\rho:Z\to X$ satisfying the properties (i), (ii), (iii) of the definition. Moreover, let $\beta:G\to\pAut(\rC(Z))$ be the dual partial action with automorphisms $\big\{\beta_t: E_{t^{-1}}\to E_t\big\}_{t\in G}$, where $E_t:=\rC_0(V_t)$ for each $t\in G$.

We claim that the $\ast$-homomorphism \[\overline{\rho}:\rC(X)\to\rC(Z);\quad\overline{\rho}(f)=f\circ\rho\]
is $G$-equivariant. To this end let $f\in D_t$ and $z\in Z$.
\[\overline{\rho}(f)(z)\neq0\implies f(\rho(z))\neq0\implies\rho(z)\in U_t\stackrel{\text{(i) of Def.~\ref{def: RFD action}}}{\implies} z\in\rho^{-1}(U_t)\subseteq V_t.\]
Thus $\overline{\rho}(D_t)\subseteq E_t$. Now if $f\in D_{t^{-1}}$, $\alpha_t(f)\in D_t$, so $\overline{\rho}(\alpha_t(f))\in E_t$. Since $\overline{\rho}(f)\in E_{t^{-1}}$, $\beta_t(\overline{\rho}(f))\in E_t$ as well. Moreover, by (iii) of Definition~\ref{def: RFD action} we have
\[\overline{\rho}(\alpha_t(f))=\alpha_t(f)\circ\rho=f\circ\theta_{t^{-1}}\circ\rho=f\circ\rho\circ\eta_{t^{-1}}=\beta_t(f\circ\rho)=\beta_t(\overline{\rho}(f)),\]
thus $\overline{\rho}$ is indeed equivariant.

If $x\in X$, find $z\in Z$ with $d(\rho(z),x)<\delta$. Then for every $f\in\Omega$
\[|f(x)|\leq |f(x)-f(\rho(z))|+|f(\rho(z))|<\varepsilon+\sup_{z\in Z}\|f(\rho(z))\|=\varepsilon+\|\overline{\rho}(f)\|_{\mathrm{u}},\]
so $\|f\|_{\mathrm{u}}\leq\varepsilon+\|\overline{\rho}(f)\|_{\mathrm{u}}$ for $f\in\Omega$.

Now let $(\phi,v):(\rC(Z),G,\beta)\too\bM_d$ be the covariant representation from Proposition~\ref{prop: finite partial systems}, and set $\varphi:=\phi\circ\overline{\rho}$. Then for all $t\in G$ and $f\in D_{t^{-1}}$,
\[v_t\varphi(f)v_t^*=v_t\phi(\overline{\rho}(f))v_t^*=\phi(\beta_t(\overline{\rho}(f)))=\phi(\overline{\rho}(\alpha_t(f)))=\varphi(\alpha_t(f)).\]
Now using the fact that $\phi$ is isometric, if $f\in\Omega$ we get
\[\unorm{\varphi(f)}=\unorm{\phi\circ\overline{\rho}(f)}=\unorm{\overline{\rho}(f)}\geq\unorm{f}-\varepsilon.\]
Thus $\alpha$ is an RFD partial action.

Conversely, suppose $\alpha$ is RFD. We then have a sequence of covariant representations:
\[(\varphi_n,v_n):(\rC(X),G,\alpha)\too\bM_{k_n}\]
with $\|\varphi_n(f)\|\to\unorm{f}$. Using Proposition~\ref{prop: cov rep into mtx} we obtain a sequences of finite sets $(Z_n)_n$, partial actions $\big(\eta_n:G\to\pSym(Z_n)\big)_n$, and strictly equivariant maps $\big(\rho_n:Z_n\to X\big)_n$ with $\varphi_n(f)=f\circ\rho_n$ for all $f\in\rC(X)$. 

Let $\delta>0$. We claim that for some $n\geq1$, $X\subseteq_\delta\rho_n(Z_n)$. This would complete the proof. Suppose not. Then for each $n$ there is a $x_n\in X$ with
\[\rho_n(Z_n)\cap B(x_n,\delta)=\emptyset.\]
By passing to a subsequence we may assume that $(x_n)_n\to x$ for some $x\in X$. Now for $n$ large enough we have $d(x,x_n)<\delta/2$. For those $n$ and $z\in Z_n$ we have
\[\delta\leq d(\rho_n(z),x_n)\leq d(\rho_n(z),x)+d(x,x_n)<d(\rho_n(z),x)+\delta/2.\]
so $d(\rho_n(z),x)>\delta_2$ for all $n\in Z_n$. Now find $f\in\rC_c(X,[0,1])$ with $f(x)=1$ vanishing outside $B(x,\delta/2)$. Then for large $n$
\[0=\unorm{f\circ\rho_n}=\unorm{\varphi_n(f)}\to\unorm{f}=1,\]
a contradiction. 
\end{proof}

It is known that a residually finite amenable group $G$ produces an RFD reduced group \cstar-algebra $\cstarl{G}$ (\cite{BekkaLouGroups}). Combining this with Propositions~\ref{prop: RFD crossed product} and~\ref{prop: cts RFD implies RFD} we arrive at the following result.

\begin{theorem}\label{thm: RFD crossed product}
Let $\theta:G\to\pHomeo(X)$ be a continuous partial action of a countable group $G$ on a compact metric space $X$. 
\begin{enumerate}[(1)]
	\item If the reduced crossed product $\rC(X)\rtimes_\lambda G$ is RFD, then the action $\theta$ is RFD.
	\item If $\theta$ is RFD and $\cstarl{G}$ is RFD, then the reduced crossed product $\rC(X)\rtimes_\lambda G$ is RFD.
	\item If $G$ is an amenable and residually finite group and $G\to\pHomeo(X)$ is an RFD continuous partial action, then $\rC(X)\rtimes_\lambda G$ is RFD.
\end{enumerate}
\end{theorem}

From Example~\ref{ex: Bernoulli shift} and Theorem~\ref{thm: RFD crossed product} we arrive at the following.

\begin{corollary}
If $G$ is countable, amenable, and residually finite, then the partial crossed product $\rC(X_G)\rtimes_\lambda G$ arising from the partial Bernoulli action is RFD. 
\end{corollary}

\subsection{Matricial Field Bundles}

We now move to quasidiagonality and the matricial filed property. Halmos introduced quasidiagonality for a single operator, and Voiculescu imported the notion to \cstar-algebra theory, obtaining a characterization of quasidiagonal \cstar-algebras in terms of external completely positive finite-rank approximations (see Chapter 7 of~\cite{BrownOzawaBook}).
By relaxing the complete positive requirements we recover the MF algebras. Matricial Field (MF) algebras algebras were introduced by Blackadar and Kirchberg in~\cite{BK97} (see also Chapter 11 of~\cite{BrownOzawaBook}). 

A \cstar-algebra $A$ is \emph{Matricial Field} if for every finite subset $\cF\subseteq A$ and every $\varepsilon>0$, there is a $d\in\bN$ and a linear map $\varphi:A\to\bM_d$ such that for all $x,y\in\cF$,
\begin{enumerate}[(i)]\label{eq: MF property}
\item $\|\varphi(x)^*-\varphi(x^*)\|<\varepsilon$,
\item $\|\varphi(xy)-\varphi(x)\varphi(y)\|<\varepsilon$, and
\item $\big|\|\varphi(x)\|-\|x\|\big|<\varepsilon$.\footnote{In the separable case we may take the approximating map $\varphi$ to be $\ast$-linear and we may replace (iii) by $\|\varphi(x)\|>\|x\|-\varepsilon$ for all $x\in\cF$. However, for our purposes the above definition is convenient and sufficient. 
}
\end{enumerate}
If we require the map $\varphi$ to be c.p.c., then $A$ is \emph{quasidiagonal} (QD).

Note that a \cstar-algebra is MF (QD) if and only if all of its separable \cstar-subalgebras are MF (QD), so it is often convenient to restrict to the separable setting. In that case, $A$ is MF (QD) if and only if there is a natural sequence $(k_n)_n$ and a sequence of linear (c.p.c.) maps $\big(\varphi_n:A\to\bM_{k_n}\big)_{n\geq1}$ satisfying: for all $a,b\in A$
\begin{enumerate}[(i)]\label{eq: MF separable property}
\item $\|\varphi_n(a)^*-\varphi_n(a^*)\|\stackrel{n\to\infty}{\too}0$,
\item $\|\varphi_n(a)\varphi_n(b)-\varphi_n(ab)\|\stackrel{n\to\infty}{\too}0$, and
\item $\|\varphi_n(a)\|\stackrel{n\to\infty}{\too}\|a\|$.
\end{enumerate}
Equivalently, a separable $A$ is MF if and only is there is faithful embedding
\[A\too\frac{\prod_{n\geq1}\bM_{k_n}}{\bigoplus_{n\geq1}\bM_{k_n}},\]
and $A$ is QD if and only if there is such an embedding with a c.p.c. lift.

Here is the pertinent definition for Fell bundles.

\begin{definition}\label{def: MF bundle}
A Fell bundle $\cB=\{B_t\}_{t\in G}$ is said to be \emph{Matricial Field} (MF) if for every $\varepsilon>0$ and finite subset $\Omega\subseteq B$ (the total space of $\cB$), there is a $d\in\bN$ and a family of linear maps $\varphi:=\big\{\varphi_t:B_t\to\bM_d\big\}_{t\in G}$ satisfying the following:
\begin{enumerate}[(i)]
\item $\|\varphi_t(b)^*-\varphi_{t^{-1}}(b^*)\|<\varepsilon$ for all $b\in \Omega\cap B_t$,
	\item  $\|\varphi_s(a)\varphi_t(b)-\varphi_{st}(ab)\|<\varepsilon$, for all $a\in \Omega\cap B_s$ and $b\in\Omega\cap B_t$,
	\item  $\big|\|\varphi_t(b)\|-\|b\|\big|<\varepsilon$, for all $b\in\Omega\cap B_t$.
\end{enumerate}
If, in addition, we require the map $\varphi_e:B_e\to\bM_d$ to be completely positive and contractive (c.p.c), the bundle $\cB$ is called Quasidiagonal (QD).
\end{definition}

Notice that if $\cB$ is MF (QD) then the unit fiber \cstar-algebra $B_e$ is also MF (QD).

\begin{proposition}\label{prop: MF alg implies MF bundle}
Let $\cB=\{B_t\}_{t\in G}$ be a Fell bundle with total space $B$. If the reduced cross-sectional \cstar-algebra $\cstarl{\cB}$ is MF (QD), then the bundle $\cB$ is MF (QD).
\end{proposition}

\begin{proof}
Suppose $\Omega\subseteq B$ is finite and $\varepsilon>0$. Set $\Omega_t:=\Omega\cap B_t$, and
\[\cF:=\bigcup_{t\in G}\big\{\lambda_t(b)\mid b\in\Omega_t\big\}\subseteq\cstarl{\cB}.\] 
Note that $\cF$ is finite. If $\cstarl{\cB}$ is MF there is a $d\in\bN$ and a linear map $\varphi:\cstarl{\cB}\to\bM_d$ satisfying the MF approximation property~\eqref{eq: MF property}.

For every $t\in G$ set
\[\varphi_t:=\varphi\circ\lambda_t:B_t\to\bM_d.\]
Clearly $\varphi$ is linear. For $b\in B_t$
\[\|\varphi_t(b)^*-\varphi_{t^{-1}}(b^*)\|=\|\varphi(\lambda_t(b))^*-\varphi(\lambda_{t^{-1}}(b^*))\|=\|\varphi(\lambda_t(b))^*-\varphi(\lambda_t(b)^*)\|<\varepsilon.\]
Now let $a\in\Omega_s$, $b\in\Omega_t$.
\begin{align*}
\|\varphi_s(a)\varphi_t(b)-\varphi_{st}(ab)\|&=\|\varphi(\lambda_s(a))\varphi(\lambda_t(b))-\varphi(\lambda_{st}(ab))\|
\\&=\|\varphi(\lambda_s(a))\varphi(\lambda_t(b))-\varphi(\lambda_{s}(a)\lambda_t(b))\|<\varepsilon.
\end{align*}
Finally, for $b\in\Omega_t$,
\[\big|\|\varphi_t(b)\|-\|b\|\big|=\big|\|\varphi(\lambda_t(b))\|-\|\lambda_t(b)\|\big|<\varepsilon,\]
thus $\cB$ is MF.

If $\cstar_\lambda(\cB)$ is QD we can choose our  $\varphi$ above to be c.p.c, and clearly $\varphi_e=\varphi\circ\lambda_e$ is c.p.c. as well.
\end{proof}

We now embark on proving converse  results of Proposition~\ref{prop: MF alg implies MF bundle} in the separable setting. For this we will need the notion of approximate representations.

\begin{definition}\label{def: approx rep}
Let $\cB=\{B_t\}_{t\in G}$ be a (unital) separable Fell Bundle over $G$, and let $\cM$ be a class of (unital) \cstar-algebras. An \emph{approximate representation} of $\cB$ in $\cM$ consists of a sequence $(M_n)_{n\geq1}$ in $\cM$ and a sequence $(\varphi_n)_{n\geq1}$, where each $\varphi_n$ denotes a family of linear maps
\[\varphi_n:=\big\{\varphi_{n,t}:B_t\to M_n\big\}_{t\in G}\]
(with $\varphi_{n,e}$ unital) satisfying: for all $s,t\in G$, and for all $a\in B_s$, $b\in B_t$,
\begin{enumerate}[(i)]
\item $\|\varphi_{n,t}(b)^*-\varphi_{n,t^{-1}}(b^*)\|\stackrel{n\to\infty}{\too}0$,
\item $\|\varphi_{n,s}(a)\varphi_{n,t}(b)-\varphi_{n,st}(ab)\|\stackrel{n\to\infty}{\too}0$. 
\end{enumerate}

An approximate representation $(\varphi_n)_{n\geq1}$ is said to be \emph{faithful} if
\[\|\varphi_{n,t}(b)\|\too\|b\|\quad\forall t\in G,\ b\in B_t.\]
\end{definition}

We will see in Proposition~\ref{prop: MF bundle} below that given a faithful approximate representation $\big(\phi_n:=\big\{\phi_{n,t}:B_t\to M_n\big\}_{t\in G}\big)_n$, we actually obtain a faithful approximate representation $\big(\varphi_n:=\big\{\varphi_{n,t}:B_t\to M_n\big\}_{t\in G}\big)_n$ satisfying
\[\varphi_{n,t}(b)^*=\varphi_{n,t^{-1}}(b^*)\quad\forall n\geq1,\ t\in G,\ b\in B_t.\] 
It is useful, however, to allow for the approximate condition (i) in Definition~\ref{def: approx rep}, to give flexibility when working with examples, for instance, in the case of MF partial actions studied below (see Definition~\ref{def: MF partial action} and Theorem~\ref{thm: MF crossed products}).

\begin{proposition}\label{prop: MF bundle}
Let $\cB=\{B_t\}_{t\in G}$ be a separable Fell Bundle and write $\cM$ for the class of matrix algebras. The following are equivalent
\begin{enumerate}[(1)]
\item There is a faithful approximate representation of $\cB$ in $\cM$;
\item There is a sequence $(\bM_{k_n})_n$ in $\cM$ and a faithful representation $$\psi:\cB\to\prod_{\omega}\bM_{k_n},$$
\item The bundle $\cB$ is MF.
\end{enumerate}
\end{proposition}

\begin{proof}
(1)$\Rightarrow$(2): Suppose $\big(\varphi_n:=\big\{\varphi_{n,t}:B_t\to\bM_{k_n}\big\}_{t\in G}\big)_n$ is a faithful approximate representation of $\cB$ in $\cM$. For every $t\in G$ and every $b\in B_t$, $\sup_{n\geq1}\|\varphi_{n,t}(b)\|<\infty$, so for any free ultrafilter $\omega$ on $\bN$ we may define a family of linear maps $\psi:=\big\{\psi_t:B_t\to\prod_{\omega}\bM_{k_n}\big\}$, where 
\[\psi_t:B_t\to\prod_{\omega}\bM_{k_n};\quad\psi_t(b)=\pi_\omega\big((\varphi_{n,t}(b))_n\big).\]
Clearly $\psi$ is a faithful representation.

(2)$\Rightarrow$(3): Consider a faithful representation \[\big\{\psi_t:B_t\to\prod_{\omega}\bM_{k_n}\big\}_{t\in G},\]
and suppose $\Omega\subseteq B$ is finite and $\varepsilon>0$. For each $t\in G$ let
$\phi_t:B_t\to\prod_{n\geq1}\bM_{k_n}$ be any linear lift of $\psi_t$. We then set 
\begin{equation}\label{eq: *-linear}
\varphi_t:B_t\to\prod_{n\geq1}\bM_{k_n};\quad\varphi_t(b):=\frac{1}{2}(\phi_t(b)+\phi_{t^{-1}}(b^*)^*).
\end{equation}
We see that $\varphi_t$ is also a linear lift of $\psi_t$ and $\varphi_{t}(b)^*=\varphi_{t^{-1}}(b^*)$ for all $t\in G$ and $b\in B_t$. Now put $\varphi_{t,n}=\pi_n\circ\varphi_t:B_t\to\bM_{k_n}$, where $\pi_n$ are the canonical coordinate projections. Now given $a\in\Omega_s$ and $b\in\Omega_t$, the sets
\[M_{a,b}=\big\{n\in\bN\mid\|\varphi_{n,s}(a)\varphi_{n,t}(b)-\varphi_{n,st}(ab)\|<\varepsilon\big\}\]
and
\[I_a=\big\{n\in\bN\mid\big|\|\varphi_{n,s}(a)\|-\|a\|\big|<\varepsilon\big\}\]
belong to the ultrafilter $\omega$. Therefore, so does the set \[T=\bigcap_{a,b\in\Omega}M_{a,b}\cap\bigcap_{a\in\Omega}I_a\in\omega\]
Choosing an $m$ in this set, the family of linear maps $\varphi:=\big\{\varphi_{m,t}:B_t\to\bM_{k_m}\big\}_{t\in G}$ is what is desired to make $\cB$ MF.

(3)$\Rightarrow$(1): The separable condition allows us to build a sequence $(\Omega_n)_{n}$ of finite subsets of the total space $B$ with certain desired properties. For each $t\in G$, let $C_t\subseteq B_t$ be a countable dense subset, and set
\[W:=\bigg\{x_1x_2\dots x_n\mid n\in\bN,\ x_j\in \bigcup_{t\in G}(C_t\cup C_t^*)\bigg\}\subseteq B.\]
Note that $W$ is countable with $W\cdot W\subseteq W$, $W^*=W$, and $W\supseteq C_t$ for every $t\in G$. Next, for each $t\in G$ put $W_t:=W\cap B_t$ and
\[D_t:=\bQ[i]\text{-}\Span(W_t).\]
Clearly $D_{t}$ is a countable linear space over the complex rationals $\bQ[i]$ with $D_t\subseteq B_t$ norm-dense. Moreover, $D_t^*=D_{t^{-1}}$ and $D_sD_t\subseteq D_{st}$ for all $s,t\in G$. In particular, $D_e$ is a $\ast$-algebra over $\bQ[i]$. Next, for each $t\in G$ we enumerate $D_t=\{d_j^t\}_{j=1}^{\infty}$. Also, since $G$ is countable we may find an increasing sequence $\{F_k\}_{k=1}^\infty$ of symmetric subsets $F_k\subseteq G$ with $F_k^2\subseteq F_{k+1}$. Now for each $k\geq1$ we have the finite set
\[\Sigma_k:=\bigcup_{t\in F_k}\{d_j^t\}_{j=1}^k.\]
\begin{itemize}
	\item Put $\Omega_1:=\Sigma_1\cup\Sigma_1^*$.
	\item $\Lambda_2:=\Sigma_2\cup\Omega_1\cup\Omega_1^2$, and $\Omega_2:=\Lambda_2\cup\Lambda_2^*$.
	\item[] $\vdots$
	\item $\Lambda_{n}:=\Sigma_{n}\cup\Omega_{n-1}\cup\Omega_{n-1}^2$, and $\Omega_{n}:=\Lambda_{n}\cup\Lambda_{n}^*$.
\end{itemize} 
The sequence $(\Omega_n)_{n\geq1}$ of finite subsets of the total space $B$ has the following properties:
\begin{enumerate}[(a)]
	\item $\Omega_1\subseteq\Omega_2\subseteq\Omega_3\cdots$,
	\item $\Omega_n^*=\Omega_n$ for all $n\geq1$, and
	\item for every $t\in G$,  $\bigcup_{n\geq1}\Omega_{n,t}\supseteq D_t$, where $\Omega_{n,t}:=\Omega_n\cap B_t$.
	\item $\Omega_{n,s}\Omega_{n,t}\subseteq\Omega_{n+1,st}$.
\end{enumerate}

Now let $(\varepsilon_n)_{n\geq1}$ be a decreasing sequence of positive numbers with $(\varepsilon_n)_{n}\to0$.
Since $\cB$ is MF every $n\geq1$ will yield a family of linear maps $\big\{\psi_{n,t}:B_t\to\bM_{k_n}\big\}_{t\in G}$ satisfying: for every $a\in\Omega_{n,s}$ and $b\in\Omega_{n,t}$,
\begin{enumerate}[(i)]
\item  $\|\psi_{n,t}(b)^*-\psi_{n,t^{-1}}(b^*)\|<\varepsilon_n$, 
\item  $\|\psi_{n,s}(a)\psi_{n,t}(b)-\psi_{n,st}(ab)\|<\varepsilon_n$, and 
\item  $\|b\|-\varepsilon_n<\|\psi_{n,t}(b)\|<\|b\|+\varepsilon_n$.
\end{enumerate}

Fix $t\in G$. By (iii) we have $\lim_{n\to\infty}\|\psi_{n,t}(b)\|=\|b\|$. This induces a $\bQ[i]$-linear isometry
\[\psi_t:D_t\too\frac{\prod_{n}\bM_{k_n}}{\oplus_{n}\bM_{k_n}};\quad\psi_t(b)=\pi\big((\psi_{n,t}(b))_n\big).\]
Properties (i) and (ii) ensure that for all $a\in D_s$ and $b\in D_t$
\begin{equation}\label{eq: psi a rep}
\psi_s(a)\psi_t(a)=\psi_{st}(ab),\ \text{ and  }\  \psi_t(b)^*=\psi_{t^{-1}}(b^*).
\end{equation}
By continuity we extend $\psi_t$ to a $\bC$-linear isometry  $\Psi_t:B_t\to\prod_{n}\bM_{k_n}/\oplus_{n}\bM_{k_n}$. A standard $\varepsilon/3$-argument along with~\eqref{eq: psi a rep} ensure that \[\bigg\{\Psi_t:B_t\too\frac{\prod_{n}\bM_{k_n}}{\oplus_{n}\bM_{k_n}}\bigg\}_{t\in G}\] is a representation of $\cB$. For each $t$, let $\phi_t:B_t\to\prod_{n\geq1}\bM_{k_n}$ be a linear lift of $\Psi_t$, and put $\varphi_t:B_t\to\prod_{n\geq1}\bM_{k_n}$ as in~\eqref{eq: *-linear}.
Now set $\varphi_{n,t}:=\pi_n\circ\varphi_t:B_t\to\bM_{k_n}$, where $\pi_n$ are the coordinate projections. We claim that the sequence
\[\bigg(\varphi_n:=\big\{\varphi_{n,t}:B_t\to\bM_{k_n}\big\}_{t\in G}\bigg)_n\]
is the faithful approximate representation we are looking for. Indeed, for $b\in B_t$ then
\[\varphi_{n,t}(b)^*=\pi_n(\varphi_t(b))^*=\pi_n(\varphi_t(b)^*)=\pi_n(\varphi_{t^{-1}}(b^*))=\varphi_{n,{t^{-1}}}(b^*).\]
Next, for $a\in B_s$ and $b\in B_t$,
\begin{align*}
\limsup_{n\to\infty}&\|\varphi_{n,s}(a)\varphi_{n,t}(b)-\varphi_{n,st}(ab)\|=\limsup_{n\to\infty}\|\pi_n(\varphi_s(a))\pi_n(\varphi_t(b))-\pi_n(\varphi_{st}(ab))\|\\
&=\limsup_{n\to\infty}\|\pi_n\big(\varphi_s(a)\varphi_t(b)-\varphi_{st}(ab)\big)\|
=\|\pi\big(\varphi_s(a)\varphi_t(b)-\varphi_{st}(ab)\big)\|\\&=\|\pi(\varphi_s(a))\pi(\varphi_t(b))-\pi(\varphi_{st}(ab))\|
=\|\Psi_s(a)\Psi_t(b)-\Psi_{st}(ab)\|=0,
\end{align*}
so $\|\varphi_{n,s}(a)\varphi_{n,t}(b)-\varphi_{n,st}(ab)\|\too0$ and $(\varphi_n)_n$ is an approximate representation. 

Now let $\omega$ be any free ultrafilter on $\bN$. The sequence $(\varphi_{n,e})_n$ induces a $\ast$-homomorphism
\[\phi_\omega:B_{e}\too\prod_{\omega}\bM_{n_k};\quad\phi_\omega(a)=\pi_\omega\big((\varphi_{n,e}(a))_n\big).\]
Since $\|\varphi_{n,e}(a)-\psi_{n,e}(a)\|\too0$ for every $a\in B_e$ we see that $\lim_{n\to\infty}\|\varphi_{n,e}(a)\|=\|a\|$ for every $a\in D_e$, thus $\phi_\omega|_{D_e}$ is isometric. It follows that $\phi_\omega$ is isometric on $B_e$, hence for all $a\in B_e$ we have
\[\lim_{\omega}\|\varphi_{n,e}(a)\|=\|a\|.\]
This holds for all free ultrafilters $\omega$, so indeed 
\[\lim_{n\to\infty}\|\varphi_{n,e}(a)\|=\|a\|.\]
Now if $b\in B_t$,
\begin{align*}
\big|\|\varphi_{n,t}(b)\|^2-\|b\|^2\big|&=\big|\|\varphi_{n,t}(b)^*\varphi_{n,t}(b)\|-\|b^*b\|\big|\\&
\leq\big|\|\varphi_{n,t^{-1}}(b^*)\varphi_{n,t}(b)\|-\|\varphi_{n,e}(b^*b)\|\big|+\big|\|\varphi_{n,e}(b^*b)\|-\|b^*b\|\big|\\
&\leq\big|\|\varphi_{n,t^{-1}}(b^*)\varphi_{n,t}(b)-\varphi_{n,e}(b^*b)\|\big|+\big|\|\varphi_{n,e}(b^*b)\|-\|b^*b\|\big|
\end{align*}
which tends to zero as $n\to\infty$. Thus $(\varphi_n)_n$ is faithful.
\end{proof}

Having shown that the MF condition in a bundle $\cB$ is necessary for the reduced cross-sectional algebra to be MF (\ref{prop: MF alg implies MF bundle}), we prove a converse under the assumption that our group is exact and has an MF reduced \cstar-algebra. Recall that a group $G$ is exact if its reduced group \cstar-algebra $\cstarl{G}$ is exact.

\begin{theorem}\label{thm: MF C*-bundle}
Let $\cB=\{B_t\}_{t\in G}$ be a separable Fell bundle over an exact group $G$. If $\cB$ and $\cstarl{G}$ are MF, then the \cstar-algebra $\cstarl{\cB}$ is MF.
\end{theorem}

\begin{proof}
Assuming that $\cB$ is an MF bundle, Proposition~\ref{prop: MF bundle} gives a faithful representation
\[\pi:\cB\to\prod_{\omega}\bM_{k_n}.\]
Corollary~\ref{cor: Fell cor} now ensures an embedding:
\[\cstarl{\cB}\hookrightarrow\bigg(\prod_{\omega}\bM_{k_n}\bigg)\otimes\cstarl{G}.\]
Since $G$ is exact and MF, $\cstarl{G}$ is exact and MF, and so the minimal tensor product $\prod_{\omega}\bM_{k_n}\otimes\cstarl{G}$ is MF by Proposition 3.6 in~\cite{RaSc2019}. Finally, the MF property passes to subalgebras, so $\cstarl{\cB}$ is MF.
\end{proof}

We can now characterize quasidiagonality in reduced cross-sectional \cstar-algebras. As with crossed products, we need the unit fiber \cstar-algebra $B_e$ to be nuclear (cf. Theorem 3.19 in~\cite{Rain2014}).

\begin{theorem}\label{thm: QD C*bundle algebras}
Let $\cB=\{B_t\}_{t\in G}$ be a separable Fell bundle with $B_e$ nuclear. If $\cstarl{G}$ is quasidiagonal and $\cB$ is MF, then $\cstarl{\cB}$ is quasidiagonal.

If $G$ is amenable and $\cB$ is MF, then $\cstarl{\cB}$ is quasidiagonal.
\end{theorem}

\begin{proof}
Since $\cstarl{G}$ is QD we know that $G$ is amenable by a result of Rosenberg (\cite{HadRQD}). By Theorem 4.7 in~\cite{ExFellAmen}, the bundle $\cB$ has Exel's approximation property and hence is amenable. Since $B_e$ is nuclear, Theorem 7.2 of~\cite{AbBuFeAmenFell} states that $\cstarl{\cB}$ is nuclear. Also, by Theorem~\ref{thm: MF C*-bundle} we know that $\cstarl{\cB}$ is MF. Finally, recall that nuclear MF algebras are quasidiagonal (\cite{BK97}).

As for the last statement, if $G$ is amenable, then $\cstarl{G}$ is quasidiagonal (\cite{SchafTWW},~\cite{TWWQD}).
\end{proof}

We now restrict our attention to partial \cstar-dynamical systems. Given such a system $\alpha:G\to\pAut(A)$, we want to express the MF (QD) condition in the associated Fell bundle $\cA_\alpha$ dynamically. This leads us to the notion of an MF (QD) partial action, generalizing the MF (QD) global actions defined in~\cite{KerrNowak} and~\cite{Rain2014}.

\begin{definition}\label{def: MF partial action}
Let $\alpha:G\to\pAut(A)$ be a partial \cstar-dynamical system with partial automorphisms $\big\{\alpha_t:D_{t^{-1}}\to D_t\big\}_{t\in G}$. We say that $\alpha$ is an MF action if:\\
for every $\varepsilon>0$, $F\subseteq G$ finite, and $\Omega\subseteq\bigcup_{t\in F}D_t$ finite, there is a $d\in\bN$, a linear map $\varphi:A\to\bM_d$, and a map $v:G\to\bM_d$ satisfying: for all $s,t\in F$, and $x,y\in\Omega$,
\begin{enumerate}[(i)]
	\item $\|\varphi(x)^*-\varphi(x^*)\|<\varepsilon$,
	\item $\|\varphi(xy)-\varphi(x)\varphi(y)\|<\varepsilon$, and
	\item $\big|\|\varphi(x)\|-\|x\|\big|<\varepsilon$.
	\item $v(e)=1$,
	\item $\sup_{t\in F}\|v_t\|\leq1$,
	\item $\|v_t^*-v_{t^{-1}}\|<\varepsilon$,
	\item $\|v_{s^{-1}}v_sv_t-v_{s^{-1}}v_{st}\|<\varepsilon$,
	\item $\|v_t\varphi(a)v_t^*-\varphi(\alpha_t(a))\|<\varepsilon$ for all $a\in\Omega_{t^{-1}}$
\end{enumerate}
If we further require the map $\varphi$ to be c.p.c. we will call the action quasidiagonal (QD).
\end{definition}

\begin{lemma}\label{lem: pert lemma}
If $\alpha:G\to\pAut(A)$ is an MF action, we can take each $v_t$ in Definition~\ref{def: MF partial action} to be a partial isometry.
\end{lemma}

\begin{proof}
Let $\varepsilon>0$, and suppose $F\subseteq G$ and $\Omega\subseteq\bigcup_{t\in F}D_t$ are finite subsets. Apply Definition~\ref{def: MF partial action} with $F$ replaced by a finite symmetric subset $K\subseteq G$ containing $F$ and $F^2$, and $0<\eta<1/8$ to be determined later. We obtain $\varphi:A\to\bM_d$ and $v:G\to\bM_d$ as in the definition, satisfying properties (i) through (viii) therein.

For every $t\in K$ we have
\[\|v_tv_t^*v_t-v_t\|\leq\|v_tv_t^*v_t-v_tv_{t^{-1}}v_t\|+\|v_tv_{t^{-1}}v_t-v_t\|\leq\|v_t^*-v_{t^{-1}}\|+\eta<2\eta. \]
Fix $t\in K$ and set $q_t=v_t^*v_t$. Then $q_t$ is self-adjoint and
\[\|q_t^2-q_t\|=\|v_t^*v_tv_t^*v_t-v_t^*v_t\|\leq\|v_tv_t^*v_t-v_t\|<2\eta.
\]
By a standard perturbation argument there is a projection $p_t\in\bM_d$ with $$\|p_t-q_t\|<4\eta.$$ Now put $w_t:=v_tp_t$. Then $w_t^*w_t$ belongs to the corner $p_t\bM_dp_t$ and
\[\|w_t^*w_t-p_t\|=\|p_tv_t^*v_tp_t-p_t\|\leq\|q_t-p_t\|<4\eta.\]
Therefore $w_t^*w_t$ is invertible in the corner $p_t\bM_dp_t$, so set $x_t:=(w_t^*w_t)^{-1/2}$ where the inverse is computed in $p_t\bM_dp_t$. A short spectral argument gives $\|x_t-p_t\|<4\eta$. Finally, set $u_t:=w_tx_t$. Then
\[u_t^*u_t=(w_tx_t)^*w_tx_t=(w_t^*w_t)^{-1/2}w_t^*w_t(w_t^*w_t)^{-1/2}=p_t,
\]
so $u_t$ is a partial isometry. Moreover,
\begin{align*}
\|u_t-v_t\|&\leq\|u_t-w_t\|+\|w_t-v_t\|
=\|w_tx_t-w_t\|+\|v_tp_t-v_t\|\\
&\leq\|w_t(x_t-p_t)\|+\|v_tp_t-v_tq_t\|+\|v_tq_t-v_t\|\\
&\leq \|x_t-p_t\|+\|p_t-q_t\|+\|v_tv_t^*v_t-v_t\|<10\eta.
\end{align*}
We have the following estimates: for all $t\in K$ we have
\[\|u_t^*-u_{t^{-1}}\|\leq\|u_t^*-v_t^*\|+\|v_t^*-v_{t^{-1}}\|+\|v_{t^{-1}}-u_{t^{-1}}\|<21\eta.\]
Also, for all $s,t\in F$
\[
u_{s^{-1}}u_su_t\stackrel{10\eta}{\approx}v_{s^{-1}}u_su_t\stackrel{10\eta}{\approx}v_{s^{-1}}v_su_t\stackrel{10\eta}{\approx}v_{s^{-1}}v_sv_t\stackrel{\eta}{\approx}v_{s^{-1}}v_{st}\stackrel{10\eta}{\approx}u_{s^{-1}}v_{st}\stackrel{10\eta}{\approx}u_{s^{-1}}u_{st},
\]
so $\|u_{s^{-1}}u_su_t-u_{s^{-1}}u_{st}\|<51\eta$.
Finally, if $C:=\max_{x\in\Omega}\|x\|$, then for $t\in F$ and $a\in\Omega_{t^{-1}}$ we have
\begin{align*}
\|u_t\varphi(a)u_t^*&-\varphi(\alpha_t(a))\|
\leq\\&\|u_t\varphi(a)u_t^*-v_t\varphi(a)u_t^*\|+
\|v_t\varphi(a)u_t^*-v_t\varphi(a)v_t^*\|
+\|v_t\varphi(a)v_t^*-\varphi(\alpha_t(a))\|\\
\leq&2\|u_t-v_t\|\|\varphi(a)\|+\eta<20\eta(C+\eta)+\eta<21\eta(1+C).
\end{align*}
It is a matter of choosing $\eta$ small enough so that $21\eta(1+C)<\varepsilon$ and $51\eta<\varepsilon$.
\end{proof}

As with global continuous actions, residually finite partial actions yield MF actions at the \cstar-level (cf. Proposition 3.3 in~\cite{KerrNowak}).

\begin{proposition}\label{prop: RF implies MF}
Let $\theta:=\big\{\theta_t:U_{t^{-1}}\to U_t\big\}_{t\in G}$ be a continuous partial action of $G$ on a compact metric space $X$. If $\theta$ is residually finite, then the induced partial \cstar-action $\alpha:G\to\pAut(\rC(X))$ is QD.
\end{proposition}

\begin{proof}
Let $e\in F\subseteq G$ be a finite symmetric subset and suppose $\Omega\subseteq\bigcup_{t\in F}\rC_0(U_t)$ is finite. Now find a $\delta>0$ such that for all $f\in\Omega$
\[x,y\in X,\ d(x,y)<\delta\implies|f(x)-f(y)|<\varepsilon/3.\]
Since the action is residually finite we obtain a finite set $Z$, a partial action of $G$ on $Z$;  $\eta=\big\{\eta_t:V_{t^{-1}}\to V_t\big\}_{t\in G}$, and a map $\rho:Z\to X$ satisfying the conditions of Definition~\ref{def: RF action}. As in Proposition~\ref{prop: cts RFD implies RFD} we have an induced $\ast$-homomorphism $\overline{\rho}:\rC(X)\to\rC(Z)$ and partial action $\beta:G\to\pAut(\rC(Z))$. Using the same argument we see that $\overline{\rho}(\rC_0(U_t))\subseteq \rC_0(V_t)$ for all $t\in F$, and that $\overline{\rho}$ is almost isometric within $\varepsilon$ on $\Omega$.

Now let $t\in F$ and $f\in\Omega\cap\rC_{0}(U_{t^{-1}})$. Then $\overline{\rho}(f)\in\rC_0(V_{t^{-1}})$ and $\beta_t(\overline{\rho}(f))\in\rC_0(V_{t})$ while $\alpha_t(f)\in \rC_0(U_t)$ and $\overline{\rho}(f)\in\rC_0(V_{t})$. If $z\in V_t$ we know that $$d\big(\theta_{t^{-1}}(\rho(z)),\rho(\eta_{t^{-1}}(z))\big)<\delta,$$
therefore
\begin{align*}
\big|\overline{\rho}(\alpha_t(f))(z)-\beta_t(\overline{\rho}(f))(z)\big|&=\big|\alpha_t(f)(\rho(z))-\overline{\rho}(f)(\eta_{t^{-1}}(z))\big|\\
&=\big|f(\theta_{t^{-1}}(\rho(z)))-f(\rho(\eta_{t^{-1}}(z)))\big|<\varepsilon.
\end{align*}
Therefore $\unorm{\overline{\rho}(\alpha_t(f))-\beta_t(\overline{\rho}(f))}<\varepsilon$. As in Proposition~\ref{prop: cts RFD implies RFD} we compose with a faithful covariant embedding
\[(\phi,v):(\rC(Z),G,\beta)\to\bM_d\]
and obtain a $\ast$-homomorphism $\varphi:=\phi\circ\rho:\rC(X)\to\bM_d$. Since $\phi$ is isometric, $\varphi$ is almost isometric within $\varepsilon$ on $\Omega$. Finally, for $t\in F$ and $f\in\Omega\cap\rC_0(U_{t^{-1}})$ we have
\begin{align*}
\|v_t\varphi(f)v_t^*-\varphi(\alpha_t(f))\|&=\|v_t\phi(\overline{\rho}(f))v_t^*-\phi(\overline{\rho}(\alpha_t(f)))\|\\
&=\|\phi(\beta_t(\overline{\rho}(f)))-\phi(\overline{\rho}(\alpha_t(f)))\|=\unorm{\beta_t(\overline{\rho}(f))-\overline{\rho}(\alpha_t(f))}<\varepsilon.
\end{align*}
\end{proof}

Using the same approach as as (3) implies (1) of Proposition~\ref{prop: MF bundle} and the perturbation techniques of Lemma~\ref{lem: pert lemma} we arrive at the following characterization of MF actions in the separable setting. 

\begin{proposition}\label{prop: separable MF action}
Let $\alpha:G\to\pAut(A)$ be a partial action with $G$ countable and $A$ separable. Then $\alpha$ is MF (QD) if and only if there is a sequence of $\ast$-linear (c.p.c) maps $\big(\varphi_n:A\to\bM_{k_n}\big)_{n\geq1}$ and a sequence of maps $\big(v_n:G\to\mathrm{PI}(\bM_{k_n})\big)_{n\geq1}$ such that: for all $x,y\in A$, $s,t\in G$, $a\in D_{t^{-1}}$,
\begin{enumerate}[(i)]
\item $\|\varphi_n(xy)-\varphi_n(x)\varphi_n(y)\|\to0$,
\item $\|\varphi_n(x)\|\to\|x\|$,
\item $v_n(e)=1_{k_n}$,
\item $\|v_n(t)^*-v_n({t^{-1}})\|\to0$,
\item $\|v_n({s^{-1}})v_n(s)v_n(t)-v_n({s^{-1}})v_n({st})\|\to0$,
\item $\|v_n(t)\varphi_n(a)v_n(t)^*-\varphi_n(\alpha_t(a))\|\to0$.
\end{enumerate}
Moreover, if $\alpha$ is MF and $A$ is nuclear, then $\alpha$ is QD.
\end{proposition}

\begin{proof}
We only prove the last statement. Assuming $\alpha$ is MF, let  $\big(\varphi_n:A\to\bM_{k_n}\big)_n$ and $(v_n)_n$ be the sequnces satisfying conditions (i) through (vi) of Proposition~\ref{prop: separable MF action}. We then get an isometric $\ast$-homomorphism 
\[\varphi:A\too\frac{\prod_{n\geq1}\bM_{k_n}}{\bigoplus_{n\geq1}\bM_{k_n}};\quad\varphi(a)=\pi\big((\varphi_n(a))_n\big).\]
Since $A$ is nuclear, $\varphi$ admits a c.p.c. lift $\psi:A\to\prod_{n\geq1}\bM_{k_n}$. Composing with the canonical projections we get c.p.c  maps $\big(\psi_n:=\pi_n\circ\psi:A\to\bM_{k_n}\big)_n$. It follows that $\|\psi_n(a)-\varphi_n(a)\|\too0$. Standard approximating arguments show that the sequences $(\psi_n)_n$ and $(v_n)_n$ satisfy the required conditions of a QD action.
\end{proof}

\begin{proposition}\label{prop: MF partial bundles}
Let $\alpha:G\to\pAut(A)$ be a partial \cstar-dynamical system with partial automorphisms $\big\{\alpha_t:D_{t^{-1}}\to D_t\big\}_{t\in G}$, and write $\cA_\alpha=\{A_t\}_{t\in G}$ for the associated Fell bundle.
\begin{enumerate}[(1)]
\item If the action $\alpha$ is MF, the bundle $\cA_\alpha$ is MF.
\item If each ideal in the partial system is unital and $\cA_\alpha$ is MF, then the action $\alpha$ is MF.
\end{enumerate}
\end{proposition}

\begin{proof}
We may assume that $A$ is separable and $G$ is countable.

(1) Let $\big(\varphi_n:A\to\bM_{k_n}\big)_n$ and  $\big(v_n:G\to\mathrm{PI}(\bM_{k_n})\big)_n$ be as in Proposition~\ref{prop: separable MF action}. These clearly induce a faithful covariant representation
\[(\varphi,v):(A,G,\alpha)\to\prod_{\omega}\bM_{k_n};\quad \varphi(a)=\pi_\omega\big((\varphi_n(a))_n\big),\ v(t)=\pi_\omega\big((v_n(t))_n\big).\]
We thus get a $\ast$ homomorphism
\[\varphi\rtimes v:A\rtimes_{\mathrm{alg}}^{\alpha}G\too\prod_{\omega}\bM_{k_n};\quad a\delta_t\mapsto\varphi(a)v_t.\]
which we compose with the $\ast$-isomorphism $\rC_c(\cA_\alpha)\cong A\rtimes_{\mathrm{alg}}^{\alpha}G$ (see~\eqref{eq: *-iso between crossed product and bundle}) to yield a $\ast$-homomorphism
\[\rC_c(\cA_\alpha)\too\prod_{\omega}\bM_{k_n};\quad (a,t)\delta_t\mapsto\varphi(a)v_t.\]
This in turn gives a representation $\pi:\cA_\alpha\to\prod_{\omega}\bM_{k_n}$;
\[\pi=\big\{\pi_t:A_t\to\prod_{\omega}\bM_{k_n}\big\}_{t\in G};\quad\pi_t((a,t))=\varphi(a)v_t.\]
Now for every $a\in A$, 
\[\|\pi_e(a,e)\|=\|\varphi(a)v_e\|=\|\varphi(a)\|=\lim_{n\to\omega}\|\varphi_n(a)\|=\|a\|,\]
so $\pi$ is faithful. Thus $\cA_\alpha$ is MF by Proposition~\ref{prop: MF bundle}.

(2) For notational convenience we will write elements in $A_t$ as $a\nu_t$ instead of the more cumbersome $(a,t)$ where $a\in D_t$.

We assume each ideal $D_t$ has unit $p_t$. In particular, $A$ is unital with unit $p_e=1_A$. Since $\alpha_t:D_{t^{-1}}\to D_t$ is a $\ast$-isomorphism, $\alpha_t(p_{t^{-1}})=p_t$. 

Now let 
\[\big(\varphi_n=\big\{\varphi_{n,t}:A_t\to\bM_{k_n}\big\}_{t\in G}\big)_{n\geq1}\] 
be an approximate representation of $\cA_\alpha$ in the class of matrix algebras. We then set
\[\phi_n:A\to\bM_{k_n};\quad\phi_n(a)=\varphi_{n,e}(a\nu_e),\]
and
\[v:G\to\bM_{k_n};\quad v_n(t)=\varphi_{n,t}(p_t\nu_t).\]
Then $v_n(e)=\varphi_{n,e}(1\nu_e)=1$, and
$\|v_n(t)\|=\|\varphi_{n,t}(p_t\nu_t)\|\too\|p_t\nu_t\|=1$. Also
\begin{align*}
\|v_n(t)^*-v_n(t^{-1})\|&=\|\varphi_{n,t}(p_t\nu_t)^*-\varphi_{n,{t^{-1}}}(p_{t^{-1}}\nu_{t^{-1}})\|\\&=\|\varphi_{n,t}(p_t\nu_t)^*-\varphi_{n,{t^{-1}}}((p_t\nu_t)^*)\|\stackrel{n\to\infty}{\too}0.
\end{align*}
Next, since $p_{s^{-1}}p_t$ is a unit for the ideal $D_{s^{-1}}\cap D_t$, and $p_{s}p_{st}$ is a unit for the ideal $D_{s}\cap D_{st}$, and $\alpha_s\big(D_{s^{-1}}\cap D_t\big)=D_{s}\cap D_{st}$ we have 
\[p_s\nu_sp_t\nu_t=\alpha_s\big(\alpha_{s^{-1}}(p_s)p_t\big)\nu_{st}=\alpha_s(p_{s^{-1}}p_t)\nu_{st}=p_{s}p_{st}\nu_{st}.\]
Also,
\[p_{s^{-1}}\nu_{s^{-1}}p_sp_{st}\nu_{st}=\alpha_{s^{-1}}\big(\alpha_s(p_{s^{-1}}p_sp_{st})\big)\nu_t=\alpha_{s^{-1}}(p_sp_{st})\nu_t=p_{s^{-1}}p_t\nu_t,\]
and, similarly, $p_{s^{-1}}\nu_{s^{-1}}p_{st}=p_{s^{-1}}p_t\nu_t$. Therefore,
\begin{align*}
\|v_n(s^{-1})&v_n(s)v_n(t)-v_n(s^{-1})v_n(st)\|
\\ = &\|\varphi_{n,s^{-1}}(p_{s^{-1}}\nu_{s^{-1}})\varphi_{n,s}(p_s\nu_s)\varphi_{n,t}(p_t\nu_t)-\varphi_{n,s^{-1}}(p_{s^{-1}}\nu_{s^{-1}})\varphi_{n,st}(p_{st}\nu_{st})\|\\
\leq &  \|\varphi_{n,s^{-1}}(p_{s^{-1}}\nu_{s^{-1}})\big(\varphi_{n,s}(p_s\nu_s)\varphi_{n,t}(p_t\nu_t)-\varphi_{n,st}(p_s\nu_sp_t\nu_t)\big)\|\\
& +\|\varphi_{n,s^{-1}}(p_{s^{-1}}\nu_{s^{-1}})\varphi_{n,st}(p_{s}p_{st}\nu_{st})-\varphi_{n,t}(p_{s^{-1}}\nu_{s^{-1}}p_sp_{st}\nu_{st})\|\\
&+\|\varphi_{n,t}(p_{s^{-1}}p_t\nu_t)-\varphi_{n,s^{-1}}(p_{s^{-1}}\nu_{s^{-1}})\varphi_{n,st}(p_{st}\nu_{st})\|\stackrel{n\to\infty}{\too}0.
\end{align*}
Looking at the linear maps $\phi_n$, for $a,b\in A$, we have $a\nu_eb\nu_e=ab\nu_e$, so
\[\|\phi_n(a)\phi_n(b)-\phi_n(ab)\|=\|\varphi_{n,e}(a\nu_e)\varphi_{n,e}(b\nu_e)-\varphi_{n,e}(ab\nu_e)\|\stackrel{n\to\infty}{\too}0.\]
Also, $(a\nu_e)^*=a^*\nu_e$, so
\[\|\phi_n(a)^*-\phi_n(a^*)\|=\|\varphi_{n,e}(a\nu_e)^*-\varphi_{n,e}(a^*\nu_e)\|\stackrel{n\to\infty}{\too}0.\]
Next, $\|\phi_n(a)\|=\|\varphi_{n,e}(a\nu_e)\|\stackrel{n\to\infty}{\too}\|a\nu_e\|=\|a\|$. Finally, if $a\in D_{t^{-1}}$ then since $(p_t\nu_t)(a\nu_e)(p_{t^{-1}}\nu_{t^{-1}})=\alpha_t(a)\nu_e$ we have
\begin{align*}
\|v_n(t)&\phi_n(a)v_n(t)^*-\phi_n(\alpha_t(a))\|\\ &\leq\|v_n(t)\phi_n(a)v_n(t)^*-v_n(t)\phi_n(a)v_n(t^{-1})\|+\|v_n(t)\phi_n(a)v_n(t^{-1})-\phi_n(\alpha_t(a))\|\\
&\leq\|v_n(t)\phi_n(a)\big(v_n(t)^*-v_n(t^{-1})\big)\|\\
&\qquad\qquad +\|\varphi_{n,t}(p_t\nu_t)\varphi_{n,e}(a\nu_e)\varphi_{n,t^{-1}}(p_{t^{-1}}\nu_{t^{-1}})-\varphi_{n,e}(\alpha_t(a)\nu_e)\|\\
&\leq\|v_n(t)\|\|\phi_n(a)\|\|v_n(t)^*-v_n(t^{-1})\|\\
&\qquad\qquad +\|\varphi_{n,t}(p_t\nu_t)\varphi_{n,e}(a\nu_e)\varphi_{n,t^{-1}}(p_{t^{-1}}\nu_{t^{-1}})-\varphi_{n,e}\big((p_t\nu_t)(a\nu_e)(p_{t^{-1}}\nu_{t^{-1}})\big)\|
\end{align*}
which tends to zero as $n\to\infty$.

Using identical perturbation techniques as in Lemma~\ref{lem: pert lemma} we can find a sequence of maps into the set of partial isometries;  $\big(w_n:G\to\mathrm{PI}(\bM_{k_n})\big)_n$, with $w_n(e)=1_{k_n}$ and  $\|v_n(t)-w_n(t)\|\to0$ for every $t\in G$. We can also find a sequence of $\ast$-linear maps $\big(\psi_n:A\to\bM_{k_n}\big)_n$ with $\|\psi_n(a)-\phi_n(a)\|\to0$ for every $a\in A$. Using a standard argument the conditions of Proposition~\ref{prop: separable MF action} are clearly satisfied by the sequences $(\psi_n)_n$ and $(w_n)_n$. Thus $\alpha$ is MF.
\end{proof}

We arrive at our final main results which characterize MF and QD crossed products arising from partial \cstar-systems respectively. These results follow immediately from Propositions~\ref{prop: MF partial bundles}, ~\ref{prop: separable MF action}, and Theorems~\ref{thm: MF C*-bundle}, and~\ref{thm: QD C*bundle algebras}.

\begin{theorem}\label{thm: MF crossed products}
Let $\alpha:G\to\pAut(A)$ be a partial \cstar-dynamical system.
\begin{enumerate}[(1)]
\item If $G$ is exact, and if $\alpha$ and $\cstarl{G}$ are MF, then the reduced crossed product $A\rtimes_\lambda^\alpha G$ is MF.
\item If each ideal of the partial system is unital, and $A\rtimes_\lambda^\alpha G$ is MF, then the action $\alpha$ is MF.
\end{enumerate}
\end{theorem}

\begin{theorem}\label{thm: QD crossed products}
Let $A$ be a nuclear \cstar-algebra and suppose $\alpha:G\to\pAut(A)$ is partial action.
\begin{enumerate}[(1)]
\item If $\alpha$ is MF and $\cstarl{G}$ is QD, then the reduced crossed product $A\rtimes_\lambda^\alpha G$ is QD.
\item If $G$ is amenable and $\alpha$ is MF, then $A\rtimes_\lambda^\alpha G$ is QD.
\item If each ideal of the partial system is unital, and $A\rtimes_\lambda^\alpha G$ is QD, then the action $\alpha$ is QD.
\end{enumerate}
\end{theorem}

Combining Proposition~\ref{prop: RF implies MF} with Theorems~\ref{thm: MF crossed products},~\ref{thm: MF C*-bundle}, and~\ref{thm: QD crossed products}, we uncover a class of MF reduced crossed products.

\begin{corollary}\label{cor: MF classical crossed products}
Let $\theta:G\to\pHomeo(X)$ be a continuous partial action of $G$ on a compact metric space $X$. 
\begin{enumerate}[(1)]
\item If $\theta$ is residually finite, $G$ is exact, and $\cstarl{G}$ is MF, then $\rC(X)\rtimes_\lambda G$ is MF.

\item If $G$ is amenable and $\theta$ is residually finite, then $\rC(X)\rtimes_\lambda G$ is QD.
\end{enumerate}
\end{corollary}

Applying Corollary~\ref{cor: MF classical crossed products} to the partial Bernoulli shift~\ref{ex: Bernoulli shift}, and appealing to Haagerup and Thorbj{\o}rnsen deep result asserting that $\cstarl{\bF_r}$ is MF (\cite{FreeGroupMF}) we arrive at the following result. 

\begin{corollary}
Let $G\to\pHomeo(X_G)$ denote the partial Bernoulli action.
\begin{enumerate}[(1)]
\item If $G$ is residually finite, exact, and $\cstarl{G}$ is MF, then $\rC(X_G)\rtimes_\lambda G$ is MF.
\item $\rC(X_{\bF_r})\rtimes_\lambda\bF_r$ is MF.
\item If $G$ is amenable and residually finite, $\rC(X_G)\rtimes_\lambda G$ is QD.
\end{enumerate}
\end{corollary}

\end{document}